\documentclass[12pt,letterpaper]{amsart}

\usepackage{graphicx}
\usepackage[utf8]{inputenc}
\usepackage{amssymb}
\usepackage{epstopdf}
\usepackage{amsmath}
\usepackage{xcolor}
\usepackage{mathtools}
\usepackage{color}
\usepackage{mathrsfs}
\usepackage{empheq}

\newtheorem{theorem}{Theorem}[section]
\newtheorem{lemma}[theorem]{Lemma}
\newtheorem{prop}[theorem]{Proposition}

\newtheorem{assumption}[theorem]{Assumption}

\newtheorem{remark}[theorem]{Remark}
\newtheorem{example}[theorem]{Example}

\numberwithin{equation}{section}

\def\ifl{\iffalse}

\def\R{{\mathbb R}}

\def\bc{\begin{center}}

\def\sgn{\mathrm{sgn}}
\def\ec{\end{center}}
\def\be{\begin{equation}}
\def\ee{\end{equation}}
\def\ba{\begin{array}}
\def\ea{\end{array}}
\def\bea{\begin{eqnarray}}
\def\eea{\end{eqnarray}}
\def\beaa{\begin{eqnarray*}}
\def\eeaa{\end{eqnarray*}}

\begin{document}
\title[Fronts generated by surface pressure]
{Global bifurcation of  water-wave fronts generated by surface pressure}

\author[]{Jifeng Chu$^{1}$ \quad Zihao Wang$^2$ \quad Yong Zhang $^3$}

\address{$^1$  School of Mathematical, Hangzhou Normal University, Hangzhou 311121, China}
\address{$^2$  Department of Mathematics, Shanghai Normal University, Shanghai 200234, China}
\address{$^3$  School of Mathematical Sciences, Jiangsu University, Zhenjiang 212013, China}

\email{jifengchu@126.com (J. Chu)}
\email{wongtzuhao@163.com (Z. Wang)}
\email{18842629891@163.com (Y. Zhang)}

\thanks{Jifeng Chu was supported by the National Natural Science Foundation of China (Grant
No. 12571168), the Science and Technology Innovation Plan of Shanghai (Grant No. 23JC1403200) and
Zhejiang Provincial Natural Science Foundation of China (No. LZ26A010006). Yong Zhang was supported by National Natural Science Foundation of China (No. 12301133), the Postdoctoral Science Foundation of China (No. 2023M741441, No. 2024T170353) and Jiangsu Education Department (No.
23KJB110007).}


\begin{abstract}
We study two-dimensional steady gravity water waves with general vorticity in a homogeneous, inviscid fluid of finite depth, subject to a surface pressure that tends to distinct constants upstream and downstream. Starting from a fixed supercritical shear flow, we construct two global branches of monotone fronts connecting distinct laminar states, with increasing and decreasing profiles corresponding to the two signs of the pressure transition. Small-amplitude fronts are first obtained via the implicit function theorem and then extended by analytic global continuation.
\end{abstract}

\maketitle

\section{Introduction}

We study two-dimensional steady gravity water waves with vorticity generated by a prescribed variation in surface pressure. The fluid is homogeneous, inviscid, and incompressible, and occupies a region of finite depth above a horizontal bed. Our interest is in monotone fronts, whose profiles approach distinct parallel shear flows upstream and downstream. Unlike solitary waves, these solutions describe a persistent change in the far-field state, sustained here by a pressure distribution with unequal limits at spatial infinity. We construct small-amplitude fronts and continue them globally.

A pioneering development in the modern theory of rotational water waves was the construction of large-amplitude periodic waves by Constantin and Strauss~\cite{constantincpam04}. Using the Dubreil--Jacotin transformation~\cite{jacotijmpa}, they formulated the non-stagnant problem as a quasilinear elliptic boundary value problem for a height function on a fixed strip. This formulation brought elliptic estimates and global bifurcation theory into a common framework for waves with prescribed vorticity. Constantin and Varvaruca~\cite{constantinarma11} subsequently used conformal coordinates to construct small-amplitude periodic waves with constant vorticity while permitting stagnation. {Constantin, Strauss, and Varvaruca~\cite{constantinacta} constructed the corresponding large-amplitude periodic waves using analytic global bifurcation theory~\cite{BuffoniToland}. Wahl\'en and Weber~\cite{wahlenduke} extended this approach to general vorticity, allowing critical layers and overhanging profiles.}
Although developed for periodic waves, these works introduced analytical methods that have also proved fundamental in the study of solitary waves and fronts.

For solitary surface waves, Beale~\cite{Beale1977} gave a local existence theory, while Mielke~\cite{mielke86} developed spatial-dynamics methods for steady flows under localized perturbations. In the rotational setting, Hur established small-amplitude existence~\cite{hur08} and symmetry results~\cite{hurmrl}; see also Groves and Wahl\'en~\cite{gw}. Wheeler~\cite{milesjma} subsequently constructed large-amplitude solitary waves with vorticity by global bifurcation. More recently, Haziot and Wheeler~\cite{HaziotWheeler2023} obtained global families for constant vorticity in a conformal formulation allowing overhanging profiles, and the authors~\cite{ChuWangZhang2026} extended this approach to general vorticity distributions. {The latter two constructions combine the center-manifold reduction} without a phase space developed by Chen, Walsh, and Wheeler~\cite{chennonlinearity} with analytic global continuation~\cite{BuffoniToland}. {Complementing these existence results, Kozlov, Lokharu, and Wheeler}~\cite{kozlovarma21} proved the nonexistence of subcritical solitary waves in unidirectional flows, placing an important restriction on the admissible far-field regimes.

{Early rigorous constructions of connecting waves arose in two-fluid systems. Amick and Turner \cite{amick89} and Mielke~\cite{mielkebook} developed small-amplitude theories for internal waves, including solitary waves and bores. Their reductions describe solitary waves as homoclinic connections to a single equilibrium and bores as heteroclinic connections between distinct equilibria. This distinction between returning to one far-field state and connecting two different states remains central when the solutions are continued beyond the small-amplitude regime.}

{For fronts, global continuation presents an additional difficulty:} a connecting profile may separate into transitions whose mutual distance tends to infinity. Chen, Walsh, and Wheeler~\cite{cwwparis,ChenWalshWheelerFronts} developed an analytic global bifurcation theory for monotone fronts of elliptic equations on infinite cylinders and applied it to large-amplitude internal bores. Their theory accounts for this loss of compactness as well as degeneration of the asymptotic linearized operators. Their subsequent work~\cite{chenvonkarman} establishes overturning along the elevation-bore family and identifies overturning or a gravity-current limit along the depression-bore family. These results concern interfaces between two fluids. {For single-phase flows, monotone bores are excluded} in the unidirectional constant-pressure setting; see~\cite{ChenWalshWheeler2018}. External forcing therefore provides a natural setting in which to seek monotone surface fronts.

Moving atmospheric pressure offers a physical motivation for considering such forcing. A concrete example is the pressure jump associated with a convective squall line during the 15 June 2006 Ciutadella rissaga event: Jansa, Monserrat, and Gomis~\cite{JansaMonserratGomis2007} reported the abrupt atmospheric pressure rise and proposed it as the forcing mechanism for the exceptional sea-level oscillations. Liu and Higuera~\cite{liujfm} analyzed wave generation by moving atmospheric disturbances using linear models, with applications to the 2022 Tonga event. Their analysis distinguishes a pressure-locked response, travelling with the disturbance, from freely propagating waves. The locked component motivates the study of steady responses in a moving frame. Both events provide physical motivation for the forcing mechanism; the unequal constant pressure limits and global monotonicity assumed here remain modelling idealizations and do not describe either complete transient event. Mathematical studies of pressure-generated steady waves include Beale~\cite{Beale1980}, Sun\cite{Sun1993}, Tuck~\cite{VandenBroeckTuck1985} and Wheeler~\cite{warma}; related forced-wave and hydraulic-transition problems are considered in~\cite{chaos,Grue2025,kejfm}.

{In the present problem, we prescribe the upstream supercritical shear flow and the shape of the pressure transition, and vary its strength. The downstream state is determined by the \emph{conjugate flow equation}. Here, a \emph{front} is a hydraulic transition between two supercritical shear flows, or equivalently a heteroclinic connection. We construct two global branches, corresponding to positive and negative pressure differences. The positive branch consists of increasing fronts of elevation, and the negative branch consists of decreasing fronts of depression. The main result, stated in Theorem~\ref{thm:main}, gives the following alternatives: the elevation branch either loses uniform unidirectionality or approaches a critical downstream state; the depression branch either loses uniform unidirectionality or has an unbounded negative pressure difference, with downstream depth tending to zero. Loss of uniform unidirectionality means that the horizontal relative velocity approaches zero along a sequence.

Small-amplitude fronts are obtained by the analytic implicit function theorem about the prescribed laminar flow. We then continue the local branches globally. The main analytical issue is to control both far fields along the continuation: local regularity alone does not prevent a transition from drifting to infinity. Monotonicity and a flow-force balance incorporating the variable pressure exclude this loss of compactness. Together with uniform regularity estimates, they yield compactness in a two-ended function space and reduce the global alternatives to the limiting behaviors described above.}

The paper is organized as follows. Section~\ref{sec:formulation} gives the height formulation and the conjugate flow equation and sets up the functional analytic framework. Section~\ref{sec:local-fronts} analyzes the linearized operators and constructs small-amplitude fronts by the analytic implicit function theorem. Section~\ref{mono} proves that monotonicity persists under continuation by means of a quotient form of the maximum principle. Section~\ref{sec:compactness} establishes the flow-force balance, uniform regularity, and compactness. Section~\ref{sec:global-continuation} constructs the global solution curves, characterizes their possible limiting behavior.Section~\ref{sec:global-continuation} constructs the global solution curves and characterizes their possible limiting behavior. Finally, Section~7 provides an explicit constant-vorticity example whose pressure profile satisfies Assumption~\ref{ass:pressure}.

{We conclude the introduction by fixing our notation for H\"older spaces.} Let $\Omega$ be a connected, open, possibly unbounded subset of $\R^n$. {We write $C_c^\infty(\bar\Omega)$ for the smooth functions on $\bar\Omega$ with compact support in $\bar\Omega$.} For {$\alpha\in(0,1)$} {and $k\in\mathbb{N}$, we denote} by $C^{k+\alpha}(\bar\Omega)$ the space of functions whose partial derivatives up to order $k$ are H\"older continuous in $\bar\Omega$ with exponent $\alpha$. We say that $u_n\to u$ in $C_{\mathrm{loc}}^{k+\alpha}(\bar\Omega)$ if $\|\varphi(u_n-u)\|_{C^{k+\alpha}(\Omega)}\to0$ for all $\varphi\in C_{c}^\infty(\bar\Omega)$.
Let $C_b^{k+\alpha}(\bar\Omega)$ be the Banach space of functions $u\in C^{k+\alpha}(\bar\Omega)$ such that $\|u\|_{C^{k+\alpha}(\Omega)}<\infty$.
When $\Omega$ is unbounded, {we denote by $C_0^k(\bar\Omega)\subset C_b^k(\bar\Omega)$ the closed subspace} of functions {whose partial derivatives up to order $k$ vanish uniformly at infinity. Writing $D^j u$ for the collection of partial derivatives of order $j$, with $D^0u=u$, this means}
\begin{equation*}
    C_0^k(\bar\Omega):=\bigg\{u\in C_b^k(\bar\Omega):\lim_{r\to\infty}\sup_{|x|=r}|D^j u(x)|=0\quad\mathrm{for~}0\leq j\leq k\bigg\}.
\end{equation*}
{We abbreviate $C_0^0$ by $C_0$. Throughout, $C$ denotes a positive constant that may change from line to line; its relevant dependencies are stated in each estimate. The two-ended H\"older spaces used for fronts are defined in Section~\ref{sec:formulation}.}


\section{{Height formulation and conjugate flows}}\label{sec:formulation}

{We first formulate the problem in height coordinates using the Dubreil--Jacotin transformation. We then determine the downstream laminar state and express the problem as an analytic equation on two-ended H\"older spaces.}
\subsection{Governing equations} {Let $t$ denote time and $(x,y)$ the horizontal and vertical coordinates. Fix an upstream depth $d>0$, and write $\eta(t,x)$ for the surface elevation relative to this depth. The fluid lies above the flat bed $y=0$ and below the free surface $y=d+\eta(t,x)$. Fix a wave speed $c>0$ and assume that the motion is steady in the frame moving at speed $c$: the surface elevation depends only on $x-ct$, while the velocity field $(u,v)$ and pressure $P$ depend only on $x-ct$ and $y$. Writing $x$ for the horizontal coordinate in this moving frame, we define the fluid domain by}
\begin{equation*}
	\Omega_\eta=\{(x,y):0<y<d+\eta(x)\}.
\end{equation*}
{Let $g>0$ denote gravitational acceleration, and normalize the fluid density to one. The velocity field and pressure satisfy the steady Euler equations}
\begin{subequations}\label{euler}
	\begin{align}
	&(u-c)u_x+vu_y=-P_x\quad&&\mathrm{in~}\Omega_\eta,\label{euler1}\\
	&(u-c)v_x+vv_y=-P_y-g\quad&&\mathrm{in~}\Omega_\eta\label{euler2},
	\end{align}
\end{subequations}
{and the incompressibility condition}
\begin{equation}
	u_x+v_y=0\quad\mathrm{in~}\Omega_\eta\tag{\ref{euler}{c}}\label{euler3},
\end{equation}
	{together with the kinematic boundary conditions}
	\begin{subequations}
		\begin{align}
		&v=0\quad&&\mathrm{on~}y=0\tag{\ref{euler}{d}}\label{euler4},\\
	&v=(u-c)\eta_x&&\mathrm{on~}y=d+\eta(x)\tag{\ref{euler}{e}}\label{euler5},
	\end{align}
	\end{subequations}
	{Let $P_{\mathrm{atm}}$ denote the constant reference atmospheric pressure and $R(x)$ its prescribed variation along the surface. The dynamic boundary condition is}
	\begin{equation}
		P=P_{\mathrm{atm}}+R(x),\quad\mathrm{on~}y=d+\eta(x).\tag{\ref{euler}{f}}\label{euler6}
	\end{equation}
	{Assume that $R$ is monotone and has finite endpoint values, denoted by $R^\pm$, satisfying}
	\begin{equation*}
		\lim_{x\to\pm\infty}R(x)=R^{\pm},
	\end{equation*}
	{where $R^+\ne R^-=0$. Denote the limiting surface elevations by $\eta_\pm$ and the limiting horizontal velocity profiles by $U_\pm$. We seek \emph{front solutions} satisfying}
	\begin{equation}
		\eta\to\eta_{\pm},\quad v\to 0,\quad u\to U_{\pm}(y),\quad \mathrm{as~}x\to\pm\infty,\tag{\ref{euler}{g}}\label{eulerasy}
	\end{equation}
	with $\eta_-\not=\eta_+$ and $U_-(y)\not=U_+(y)$. {The limits as $x\to-\infty$ and $x\to+\infty$ define the \emph{upstream} and \emph{downstream} flows, respectively. We choose the reference level so that $\eta_-=0$; the upstream depth is therefore $d$, and the downstream depth is $d_+=d+\eta_+$.}

	 {Throughout the paper, we impose the unidirectionality condition}
	\begin{equation*}
		u-c<0\quad\mathrm{in~}\Omega_\eta,
	\end{equation*}
	{which excludes stagnation points in the fluid domain. Under this assumption, the} \emph{Froude number} is defined as
	\begin{equation*}
		\frac{1}{F^2}=g\int_{0}^{d}\frac{dy}{(c-U_-(y))^2},
	\end{equation*}
{The upstream flow is called \emph{supercritical} if $F>1$.}

	{We work throughout with a fixed supercritical upstream flow. The condition $F>1$ provides the coercivity used in the linear analysis below; moving-pressure responses in related physical regimes are discussed in~\cite{liujfm}.}
	{For the upstream shear flow, define the Bernoulli parameter $\lambda$ and the total head $Q$ by}
	\begin{equation}\label{upbernoulli}
	\lambda=(c-U_-(d))^2,\quad Q=\lambda+2gd.
	\end{equation}

	{Incompressibility gives a stream function} $\psi$ such that
	\begin{equation*}
		\psi_x=-v,\quad\psi_y=u-c.
	\end{equation*}
	{The kinematic boundary conditions \eqref{euler4} and \eqref{euler5} imply that $\psi$ is constant on the bed and on the free surface. We normalize $\psi$ to be zero on the free surface and write $m>0$ for its value on the bed.}
	In the fluid domain, we define the vorticity as
	\begin{equation*}
		\omega=v_x-u_y,
	\end{equation*}
	{and taking the curl of the Euler equations shows that vorticity is constant along streamlines. Thus there is a vorticity function $\gamma:\R\to\R$ such that~\cite{c-book}}
	\begin{equation*}
		-\Delta\psi=\gamma(\psi)
	\end{equation*}
	and define the integrated vorticity $\Gamma(p):=\int_0^p\gamma(-s)\,ds$ for $p\in[-m,0]$. Bernoulli's law states that the quantity
	\begin{equation*}
		E=\frac{1}{2}|\nabla\psi|^2+gy+P-\Gamma(-\psi)
	\end{equation*}
	is constant throughout the fluid. Combining Bernoulli's law with the dynamic boundary condition gives the following free-boundary problem for the stream function:
	\begin{subequations}\label{streameq}
		\begin{align}
			\Delta\psi&=-\gamma(\psi)\quad&&\mathrm{in~}\Omega_\eta,\label{streameq1}\\
			\psi&=m&&\mathrm{on~}y=0,\label{streameq2}\\
			\psi&=0&&\mathrm{on~}y=d+\eta(x),\label{streameq3}\\
			\frac{1}{2}|\nabla\psi|^2+gy&=\frac{Q}{2}-R&&\mathrm{on~}y=d+\eta(x)\label{streameq4},
		\end{align}
	\end{subequations}
together with the asymptotic condition
	\begin{equation}
		\eta\to\eta_{\pm},\quad\psi_x\to0,\quad\psi_y\to U_{\pm}(y)-c,\quad\mathrm{as~}x\to\pm\infty.\tag{\ref{streameq}{e}}\label{streameq5}
	\end{equation}
	
	{We use the pressure strength as a parameter and prescribe}
	\begin{equation}\label{eq:pressureform}
		 R=\beta r(x),
		\end{equation}
	 {where the normalized profile $r$ is fixed. Its monotonicity, normalization, and decay are specified in the following assumption.}
	\begin{assumption}[Monotone pressure step]\label{ass:pressure}
		The pressure is given by \eqref{eq:pressureform}, where
		\(r\in C_b^{4+\alpha}(\R)\), \(r'(q)>0\), and there are \(C,\vartheta>0\) such that
		\begin{equation*}
			\sum_{j=0}^4|{\partial_q^j r(q)-\partial_q^j r_\pm}|
			\le Ce^{-\vartheta |q|}\quad\text{for }\ \pm q\ge1,
			\qquad\lim_{q\to-\infty}r= r_-=0,\quad \lim_{q\to\infty}r=r_+=1,
		\end{equation*}
		where \(\partial_q^j r_\pm=0\) for \(j\ge1\).
	\end{assumption}

	{The unidirectionality condition allows us to apply the Dubreil--Jacotin transformation. Set}
	\begin{equation*}
		q=x,\quad p=-\psi,
	\end{equation*}
	{Thus $q$ is the horizontal coordinate and $p$ is the negative stream-function coordinate. The transformation maps $\Omega_\eta$ onto the fixed infinite strip}
	\begin{equation*}
		\Omega:=\{(q,p)\in\R^2:-m<p<0\},
	\end{equation*}
	with the top boundary
	\begin{equation*}
	\mathcal	T:=\{(q,p)\in\R^2:p=0\},
	\end{equation*}
	and the bottom boundary
	\begin{equation*}
	\mathcal	B:=\{(q,p)\in\R^2:p=-m\}.
	\end{equation*}
	
	We define the height function $h(q,p)=y$ in $\Omega$. {As shown in} \cite{constantincpam04,constantinstraussarma11,warma}, {the velocity field and surface elevation are recovered from}
	\begin{equation*}
		u=c-\frac{1}{h_p},\quad v=-\frac{h_q}{h_p},\quad \eta(x)=h(q,0)-d.
	\end{equation*}
	Moreover, the height function $h$ solves
	\begin{subequations}\label{heighteq}
		\begin{align}
		\left(-\frac{1+h_q^2}{2h_p^2}+\Gamma\right)_p+\left(\frac{h_q}{h_p}\right)_q& =0\quad&&\mathrm{in~}\Omega,\label{heighteq1}\\
		\frac{1+h_q^2}{2h_p^2}+gh&=\frac{Q}{2}-R\quad&&\mathrm{on~}\mathcal T,\label{heighteq2}\\
		h&=0&&\mathrm{on~}\mathcal B.
		\end{align}
	\end{subequations}
	{Let $H_\pm$ denote the laminar height profiles at the two ends, and let $D$ denote the gradient in $(q,p)$. The asymptotic conditions are}
	\begin{equation*}
		\lim_{q\to\pm\infty}h(q,p)=H_{\pm}(p),\quad\lim_{q\to\pm\infty}Dh(q,p)=DH_{\pm}(p),
	\end{equation*}
	{uniformly in $p$. We write $H:=H_-$ for the upstream profile, which is given by}
	\begin{equation}\label{upflow}
		H(p)=\int_{-m}^p\frac{ds}{\sqrt{\lambda+2\Gamma(s)}},
	\end{equation}
	with\begin{equation*}
		\lambda>-2\min_{[-m,0]}\Gamma.
	\end{equation*}	
	{Writing $H(p;\lambda)$ when we wish to display the parameter dependence, \eqref{upbernoulli} gives}
	\begin{equation*}
	\lambda=\frac{1}{H_p^2(0;\lambda)}={(c-U_-(d))^2}.
\end{equation*}
	{For a general laminar parameter $\mu>-2\min\Gamma$, let $H(p;\mu)$ be the profile in \eqref{upflow} with $\lambda$ replaced by $\mu$. Define its depth function $\mathcal D$ and Bernoulli function $\mathcal Q$ by}
	\begin{equation*}
	\mathcal{D}(\mu):=\int_{-m}^0\frac{d p}{\sqrt{\mu+2\Gamma(p)}},
		\qquad
		\mathcal{Q}(\mu):=\frac\mu2+g\mathcal D(\mu),
	\end{equation*}
	{and the corresponding Froude number by}
	\begin{equation}\label{eq:F-mu}
		F^{-2}(\mu):=g\int_{-m}^0\frac{d p}{(\mu+2\Gamma(p))^{3/2}}=g\int_{-m}^0 H_p^3(p;\mu)dp.
	\end{equation}
	{Direct differentiation gives}
	\begin{equation}\label{eq:B-derivative}
	\mathcal{Q}'(\mu)=\frac12\bigl(1-F^{-2}(\mu)\bigr),
		\qquad
		\mathcal{D}'(\mu)=-\frac1{2gF^2(\mu)}.
	\end{equation}
	{The function $F^{-2}(\mu)$ is strictly decreasing on $(-2\min_{[-m,0]}\Gamma,\infty)$, with $F^{-2}(\mu)\to\infty$ as $\mu\downarrow-2\min\Gamma$ and $F^{-2}(\mu)\to0$ as $\mu\to\infty$. Hence there is a unique $\lambda_{\rm cr}$ such that}
	\begin{equation*}
		F(\lambda_{\rm cr})=1.
	\end{equation*}
{The upstream flow is supercritical precisely when $\lambda>\lambda_{\rm cr}$. Let $\lambda_+$ denote the Bernoulli parameter of the downstream laminar flow. Passing to the two limits in the surface Bernoulli condition and using conservation of the total head $Q$, we obtain}
	\begin{equation*}
	\mathcal{Q}(\lambda)=Q/2,\quad \mathcal{Q}(\lambda_+)+\beta=Q/2.
	\end{equation*}
	{Eliminating $Q$ gives the \emph{conjugate flow equation}}
	\begin{equation*}
		\mathcal{Q}(\lambda_+)+\beta=\mathcal{Q}(\lambda).
	\end{equation*}
	{By \eqref{eq:B-derivative}, $\mathcal Q'(\mu)>0$ for $\mu>\lambda_{\rm cr}$. This monotonicity determines the supercritical downstream state uniquely.}
	\begin{lemma}[{The conjugate downstream state}]\label{lem:downstream-state}
		{Define the critical pressure strength by}
		\[\beta_{\rm cr}=\mathcal Q(\lambda)-\mathcal{Q}(\lambda_{\rm cr}).\]
		{For every $\beta<\beta_{\rm cr}$, there is a unique}
		\(\lambda_+(\beta)\in(\lambda_{\rm cr},\infty)\) such that
		\begin{equation}\label{eq:conjugate}
			\mathcal{Q}(\lambda_+(\beta))+\beta=\mathcal{Q}(\lambda).
		\end{equation}
		{It depends real-analytically on $\beta$. Denote the corresponding downstream depth and Froude number by $d_+(\beta)$ and $F_+(\beta)$, respectively; then}
		\begin{equation*}
			d_+(\beta)=\mathcal D(\lambda_+(\beta)),\qquad
			F_+(\beta)=F(\lambda_+(\beta)).
		\end{equation*}
		Moreover,
		\begin{equation}\label{eq:depth-derivative}
			d_+'(\beta)=\frac{1}{g(F_+^2(\beta)-1)}>0,
			\qquad d_+'(0)=\frac1{g(F^2-1)}.
		\end{equation}
		As \(\beta\uparrow\beta_{\rm cr}\), \(F_+(\beta)\downarrow1\); as
		\(\beta\to-\infty\), \(\lambda_+(\beta)\to\infty\) and \(d_+(\beta)\to0\).
	\end{lemma}
	\begin{proof}
		{On $(\lambda_{\rm cr},\infty)$, the function $\mathcal Q$ is strictly increasing and tends to infinity. Thus \eqref{eq:conjugate} has a unique solution for each $\beta<\beta_{\rm cr}$, and the analytic implicit function theorem gives its real-analytic dependence on $\beta$.}  Differentiating
		\eqref{eq:conjugate} and using \eqref{eq:B-derivative} gives
		\[
		\lambda_+'(\beta)=-\frac{2}{1-F_+^{-2}(\beta)}.
		\]
		Combining this identity with the formula for \(\mathcal D'\) proves
		\eqref{eq:depth-derivative}.  The two limiting statements follow from monotonicity
		and the definitions of \(\mathcal Q\), \(F\), and \(\mathcal D\).
	\end{proof}
	Notice that \eqref{eq:conjugate} is equivalent to the endpoint Bernoulli identity
	\begin{equation}\label{eq:endpoint-lambda}
		\lambda_+(\beta)
		=\lambda+2g\bigl(d-d_+(\beta)\bigr)-2\beta.
	\end{equation}
	Also, the downstream Froude number is
	\begin{equation*}
		\frac1{F_+^2(\beta)}
		=g\int_0^{d_+(\beta)}\frac{d y}{(c-U_+(y))^2}
		=g\int_{-m}^0H_p(p;\lambda_+(\beta))^3d p.
	\end{equation*}

	\subsection{{Functional analytic formulation}}Choose a smooth cutoff \(\chi_+\in C^\infty(\R)\) with
	\begin{equation*}
		\chi_+(q)=\begin{cases}
			0\quad&\mathrm{on~}(-\infty,-1],\\
			1&\mathrm{on~}[1,\infty),
		\end{cases}
	\end{equation*}
	and put \(\chi_-=1-\chi_+\).  For an integer
	\(k\ge0\), let
	\[
	\begin{aligned}
		C_0^k(\overline\Omega):=\bigl\{u\in C_b^k(\overline\Omega):{}&
		\partial_q^i	\partial_p^j u(q,p)\to0\ \text{uniformly in }p~\text{as }q\to\pm\infty,\quad i+j\le k\bigr\}.
	\end{aligned}
	\]
	Define the two-ended H\"older space by
	\begin{equation}\label{eq:two-ended-space}
		\begin{aligned}
			C_\infty^{k+\alpha}(\overline\Omega):=
			\bigl\{\chi_-u_-+\chi_+u_++u_0:{}&
			u_\pm\in C^{k+\alpha}([-m,0]),~u_0\in C_b^{k+\alpha}(\overline\Omega)
			\cap C_0^k(\overline\Omega)\bigr\}.
		\end{aligned}
	\end{equation}
	The endpoint profiles are unique, and the sum of their norms and the
	\(C_b^{k+\alpha}\)-norm of the remainder makes
	\eqref{eq:two-ended-space} a Banach space.  The analogous
	definition uses scalar endpoint limits  on \(\mathcal T\).  Different choices of \(\chi_+\) will give
	equivalent norms.
	This norm is equivalent to the inherited \(C_b^{k+\alpha}\)-norm:
	each endpoint norm is bounded by \(\|u\|_{C_b^{k+\alpha}}\), and the fixed
	cutoffs control the norm of the remainder. The same holds on \(\mathcal T\).
	
	Set
	\begin{align*}
		\mathscr X_\infty&:=\{w\in C_\infty^{3+\alpha}(\overline\Omega):w=0\text{ on }\mathcal B\},\\
		\mathscr Y_\infty&:=C_\infty^{1+\alpha}(\overline\Omega)
		\times C_\infty^{2+\alpha}(\mathcal T).
	\end{align*}
	and
	\begin{align*}
		\mathscr X_b&:=\{w\in C_b^{3+\alpha}(\overline\Omega):w=0\text{ on }\mathcal B\},\\
		\mathscr Y_b&:=C_b^{1+\alpha}(\overline\Omega)
		\times C_b^{2+\alpha}(\mathcal T).
	\end{align*}
	{We write $\mathbf n=(0,1)$ for the outward unit normal to $\mathcal T$.}

	{Writing $h=H(p)+w(q,p)$, we obtain the following problem for the height perturbation:}
	\begin{equation}\label{wpde}
		\begin{aligned}
			&\left( -\frac{1 + w_q^2}{2(H_{p} + w_p)^2} + \Gamma \right)_p + \left( \frac{w_q}{H_{p} + w_p} \right)_q=0\quad&&\mathrm{in~}\Omega,\\
			&\left( \frac{1 + w_q^2}{2(H_{p} + w_p)^2} + gw - \frac{\lambda}{2} + R(q, \beta) \right)=0&&{\mathrm{on~}\mathcal T},\\
			&w=0&&{\mathrm{on~}\mathcal B},
		\end{aligned}
	\end{equation}
	{For $w\in\mathscr X_\infty$, let $w_\pm$ denote its endpoint perturbations. The prescribed endpoints satisfy}
$$\lim_{q\to\infty}w(q,p)=w_+=H(\cdot;\lambda_+(\beta))-H,\quad\mathrm{and~}\lim_{q\to-\infty}w=0$$
	uniformly in $p\in[-m,0]$.

	{To construct small-amplitude fronts, we formulate the problem on the open set}
	\begin{equation}\label{eq:U-open}
		\mathscr{U}:=\{(\beta,w)\in\R\times\mathscr X_\infty:\beta<\beta_{\rm cr},\quad \inf_\Omega(H_p+w_p)>0\}.
	\end{equation}
	{We write the problem as the nonlinear equation $\mathcal F(\beta,w)=0$, where}
	\begin{equation*}
		\mathcal{F}=(\mathcal{F}_1,\mathcal{F}_2):\mathscr{U}\to \mathscr Y_\infty,
	\end{equation*}
	where
	\begin{equation}\label{eq:nonlinear-operator}
		\begin{aligned}
	&\mathcal{F}_1(w)=\left( -\frac{1 + w_q^2}{2(H_{p} + w_p)^2} + \Gamma \right)_p + \left( \frac{w_q}{H_{p} + w_p} \right)_q,\\
	&\mathcal{F}_2(\beta,w)=\left( \frac{1 + w_q^2}{2(H_{p} + w_p)^2} + gw - \frac{\lambda}{2} + R(q, \beta) \right) \bigg|_\mathcal T.
		\end{aligned}
	\end{equation}
\begin{lemma}[Analytic formulation]\label{lemma:analytic-map}
	The mapping
	\[
	\mathcal F=(\mathcal F_1,\mathcal F_2):\mathscr U\longrightarrow\mathscr {Y}_\infty
	\]
	is real analytic.  Its zero set is in one-to-one correspondence with classical
	solutions of \eqref{heighteq} having two laminar limits. {For such a solution, denote the endpoint height profiles by $K_\pm:=H+w_\pm$.} If a zero belongs to the
	connected component of \((0,0)\) on which both limiting states are supercritical,
	then
	\begin{equation}\label{eq:endpoint-profiles}
		{K_-(p)}=H(p;\lambda),\qquad
		{K_+(p)}=H(p;\lambda_+(\beta)).
	\end{equation}
\end{lemma}

\begin{proof}
	{The two-ended H\"older spaces are closed under pointwise multiplication. The expressions in \eqref{eq:nonlinear-operator} involve products of derivatives of $w$ of order at most two, powers of $(H_p+w_p)^{-1}$, and the term $\beta r$.}
	
	{At each point of $\mathscr U$, the denominator $H_p+w_p$ is bounded away from zero, and this bound persists in a neighborhood. Inversion is real analytic on the open set of uniformly positive functions in the H\"older algebra and preserves the endpoint limits. Hence $\mathcal F:\mathscr U\to\mathscr Y_\infty$ is real analytic.}
	
	For a zero, pass to either limit \(q\to\pm\infty\).  All \(q\)-derivatives vanish,
	{and the endpoint profile $K_\pm:=H+w_\pm$, with $K_{p,\pm}:=\partial_pK_\pm$, satisfies}
	\[
	\left(-\frac1{2K_{p,\pm}^2}+\Gamma\right)_p=0,\qquad {K_\pm(-m)=0}.
	\]
	Thus \(K_\pm(p)=H(p;\mu)\) {for a suitable parameter $\mu>-2\min\Gamma$ at each end. Denote these two parameters by $\mu_\pm$ and the endpoint pressures by $R_\pm(\beta)$. The limiting boundary condition is}
	\[
	\mathcal Q(\mu)+R_\pm(\beta)=\mathcal Q(\lambda),
	\qquad R_-(\beta)=0,\quad R_+(\beta)=\beta.
	\]
	On a supercritical component, \(\mathcal Q'(\mu)>0\).  Continuity from \((0,0)\) and
	{Lemma~\ref{lem:downstream-state} imply} \(\mu_-=\lambda\) and
	\(\mu_+=\lambda_+(\beta)\).  This proves \eqref{eq:endpoint-profiles}.
\end{proof}
{For solutions on this supercritical component, we write $K_\pm(p)=H(p)+w_\pm(p)$. Thus $K_-=H(\cdot;\lambda)$ is the upstream profile and $K_+=H(\cdot;\lambda_+(\beta))$ is the downstream profile.}


	\section{{Linear analysis and small-amplitude fronts}}\label{sec:local-fronts}

{We analyze the linearization of the nonlinear operator introduced in Section~\ref{sec:formulation}. The supercriticality of the limiting flows yields invertibility of the limiting operators and the Fredholm property. We then apply the analytic implicit function theorem to construct small-amplitude monotone fronts.}

\subsection{{Linearized operators}}
	{The Fr\'echet derivative $\mathscr L=\mathcal F_w(\beta,w):\mathscr X_\infty\to\mathscr Y_\infty$ is given by}
	\begin{equation}\label{eq:fww}
		\begin{aligned}
			{\mathscr L_1\varphi}&=\mathcal{F}_{1w}(w)\varphi=
			\left( \frac{1 + w_q^2}{(H_{p} + w_p)^3} \varphi_p - \frac{w_q}{(H_{p} + w_p)^2} \varphi_q \right)_p
			+\left( \frac{1}{H_{p} + w_p} \varphi_q - \frac{w_q}{(H_{p} + w_p)^2} \varphi_p \right)_q\\
			{\mathscr L_2\varphi}&=\mathcal{F}_{2w}(\beta,w)\varphi=-\frac{1 + w_q^2}{(H_{p} + w_p)^3} \varphi_p + \frac{w_q}{(H_{p} + w_p)^2} \varphi_q + g \varphi.
		\end{aligned}
	\end{equation}
	{For each fixed $(\beta,w)\in\mathscr U$, the operator $\mathscr L_1$ is uniformly elliptic and $\mathscr L_2$ is uniformly oblique.}

	{In particular, let $\mathcal L:=\mathcal F_w(0,0)$ denote the laminar linearization. Its components are}
	\begin{equation*}
		\mathcal{F}_{1w}(0)\varphi = \left( \frac{1}{H_{p}^3}\varphi_p \right)_p + \left( \frac{1}{H_{p}}\varphi_q \right)_q, \quad \mathcal{F}_{2w}(0)\varphi = -\frac{1}{H_{p}^3}\varphi_p + g\varphi\bigg|_{\mathcal T}.
	\end{equation*}
	{The following weighted trace inequality supplies the coercivity of the laminar linearization.}
	\begin{lemma}[Weighted trace inequality]\label{lem:trace}
		Let $K(p)=H(p;\mu)$. If \(\zeta\in H^1(-m,0)\) and \(\zeta(-m)=0\), then
		\begin{equation}\label{eq:weighted-trace}
			g\zeta(0)^2
			\le F^{-2}(\mu)\int_{-m}^0K_p^{-3}\zeta_p^2d p.
		\end{equation}
		Consequently, when \(F(\mu)>1\), we have
		\begin{equation}\label{eq:transverse-coercivity}
			\int_{-m}^0K_p^{-3}\zeta_p^2d p-g\zeta(0)^2
			\ge(1-F^{-2}(\mu))\int_{-m}^0K_p^{-3}\zeta_p^2d p.
		\end{equation}
	\end{lemma}
	\begin{proof}
		{Since $\zeta(-m)=0$, the weighted Cauchy--Schwarz inequality gives}
		\begin{equation*}
			g\zeta^2(0)=g\left(\int_{-m}^0\zeta_pdp\right)^2\leq g\left(\int_{-m}^0K_p^3d p\right) \left(\int_{-m}^0K_p^{-3}\zeta_p^2d p\right).
		\end{equation*}
		{Using \eqref{eq:F-mu} gives \eqref{eq:weighted-trace}, and rearranging yields \eqref{eq:transverse-coercivity}.}
	\end{proof}

	{On the unbounded strip, the Fredholm property is governed by the limiting operators~\cite{vol}. Passing to the limits $q\to\pm\infty$ in \eqref{eq:fww} gives $\mathscr L_\pm=(\mathscr L_{\pm,1},\mathscr L_{\pm,2})$, where}
	\begin{equation*}
		\begin{aligned}
		{\mathscr L_{\pm,1}}\varphi=\left(\frac{1}{K_{p,\pm}^3}\varphi_p\right)_p+\left(\frac{1}{K_{p,\pm}}\varphi_q\right)_q,\\
		{\mathscr L_{\pm,2}}\varphi=-\frac{1}{K_{p,\pm}^3}\varphi_p+g\varphi\bigg|_{\mathcal T}.
		\end{aligned}
	\end{equation*}
	{Separation of variables in the limiting equation $\mathscr L_\pm\varphi=0$ leads to the one-dimensional Sturm--Liouville problem}
	\begin{equation}\label{eq:limsl}
		\begin{aligned}
			&(K_{p,\pm}^{-3}\phi_p)_p=\sigma K_{p,\pm}^{-1}\phi\qquad\mathrm{in~}(-m,0)\\
			&-K_{p,\pm}^{-3}(0)\phi_p(0)+g\phi(0)=0,\quad\phi(-m)=0.
		\end{aligned}
	\end{equation}
	{Classical spectral theory~\cite{laxbook} gives a sequence of eigenvalues $\sigma_0^\pm(K)>\sigma_1^\pm(K)>\cdots$ tending to $-\infty$. The superscript $\pm$ distinguishes the two ends; for a single laminar profile $K$, we write $\sigma_0(K)$ for its principal eigenvalue. Its sign is determined by the Froude number.}
	
	\begin{lemma}[{Sign of the principal eigenvalue}]
	 {For a laminar flow $K(p)=H(p;\mu)$ with Froude number $F(\mu)$, we have $F(\mu)\geq1$ if and only if $\sigma_0(K)\leq0$. In particular, $F(\mu)=1$ if and only if $\sigma_0(K)=0$.}
	\end{lemma}
	\begin{proof}
	{The principal eigenvalue $\sigma_0(K)$ has the variational characterization}
	\begin{equation}\label{eq:efnu0}
		\sigma_0(K)=\sup_{\phi\in H^1(-m,0),\phi(-m)=0,\phi\not\equiv0}\frac{-\int_{-m}^0K_p^{-3}\phi_p^2dp+g\phi^2(0)}{\int_{-m}^0K_p^{-1}\phi^2dp}.
	\end{equation}
	{Direct integration shows that}
	\[\phi(p)=\int_{-m}^p K_p^3(s)ds\]
	solves \eqref{eq:limsl} with $\sigma=0$ if and only if $F(\mu)=1$. On the other hand, for $F(\mu)>1$,  by Lemma \ref{lem:trace}, we have
		\begin{equation*}
			g\phi^2(0)-\int_{-m}^0K_p^{-3}\phi_p^2dp\leq-(1-F^{-2}(\mu))\int_{-m}^0K_p^{-3}\phi_p^2dp<0.
		\end{equation*}
		For $F(\mu)<1$, taking $\phi(p)=\int_{-m}^pK_p^3(s)ds$ in \eqref{eq:efnu0} gives $\sigma_0(K)>0$.
		{This proves the claimed equivalences.}
		\end{proof}

	{The positive zeroth-order term in the top boundary condition prevents a direct application of the usual maximum principle. A quotient transformation using a positive comparison function restores the required sign. We first state the argument for functions that vanish at both ends; a translation argument below extends its use to bounded solutions of the limiting equations.}

	{For a nonnegative integer $j$, we use $\mathbb M_2(C_b^j(\overline\Omega))$ to denote the space of $2\times2$ matrix fields with entries in $C_b^j(\overline\Omega)$; the same convention applies to H\"older regularity.}

\begin{lemma}[Comparison function via the quotient transformation]\label{lem:comparison-function}
		Let $A=A^T \in \mathbb  M_2(C_b^1(\overline{\Omega}))$ be uniformly positive definite and {define the interior and boundary operators by}
		\begin{equation*}
		Lu:=\mathrm{div}(A\nabla u)\mathrm{~on~}\Omega,\quad Bu:=-{\mathbf n\cdot}(A\nabla u)+gu\mathrm{~on~}\mathcal T.
		\end{equation*}
		{Suppose that there exist a positive comparison function $\phi\in C_b^2(\overline\Omega)$ and constants $c_0,c_1>0$ such that}
		\begin{equation*}
			\phi\ge c_0,\qquad L\phi\le-c_1,
			\qquad B\phi\le-c_1.
		\end{equation*}
		{Let $f$ be a nonpositive boundary datum on $\mathcal T$. If $u\in C_b^2(\overline\Omega)\cap C_0(\overline\Omega)$ satisfies}
		\begin{equation*}
			\begin{aligned}
				&Lu=0\quad\mathrm{in~}\Omega,\\
				&Bu=f\leq0 \quad\mathrm{on~}\mathcal T,\\
				&u=0\quad\mathrm{on~}\mathcal B,
			\end{aligned}
		\end{equation*}
		then $u\geq0$ on $\Omega\cup \mathcal{T}$. In particular, if $f<0$ on $\mathcal{T}$, then $u>0$ on $\Omega\cup \mathcal{T}$.
 	\end{lemma}
	\begin{proof}
		Write \(u=\phi v\). {Direct expansion shows that $v\in C_b^2(\overline\Omega)\cap C_0(\overline\Omega)$ satisfies}
		\begin{subequations}\label{eq:quotient-equation}
			\begin{align}
			\operatorname{div}(A\nabla v)
				+2A\nabla(\log\phi)\cdot\nabla v
				+\frac{L\phi}{\phi}v&=0\quad  &&\text{in }\Omega,\label{eq:quotient-interior}\\
				-\mathbf n\cdot A\nabla v+\frac{B\phi}{\phi}v&=\frac{f}{\phi} &&\text{on }\mathcal T,\label{eq:quotient-boundary}\\
				v&=0              &&\text{on }\mathcal B.\label{eq:quotient-bed}
			\end{align}
		\end{subequations}
	{Since $L\phi/\phi<0$, the maximum principle applies to the interior equation. Suppose that $v$ takes a negative value. Because $v\to0$ at both ends, it attains a negative minimum. The strong maximum principle excludes an interior minimum, and $v=0$ excludes a minimum on $\mathcal B$. At a minimum on $\mathcal T$, the tangential derivative vanishes and the outward normal derivative is nonpositive.}  Hence
		\(-\mathbf n\cdot A\nabla v\ge0\), while
		\((B\phi/\phi)v>0\).  This contradicts
		\eqref{eq:quotient-boundary}, whose right-hand side is non-positive.  Thus
		\(u\ge0\).  The strong maximum principle gives positivity in \(\Omega\) unless
		\(u\equiv0\); the latter is impossible when \(f<0\).  A zero on \(\mathcal T\) is
		excluded directly from {the boundary equation and Hopf's lemma}.
	\end{proof}
	 {For a laminar profile $K=H(\cdot;\mu)$, we also write the Froude number $F(K):=F(\mu)$. When $F(K)>1$, define the comparison function}
	\begin{equation}\label{eq:explicit-comparison}
	\Phi[K](p)=1+C_1\int_{-m}^pK_p^3(t)dt-C_2\left(\int_{-m}^pK_p^3(t)\right)^2,
	\end{equation}
	{where $F=F(K)>1$ and $C_2>0$. Choose $C_1>\frac{2C_2}{gF^2}$ sufficiently large that} {$\Phi\geq1$} and {the right-hand side of} \eqref{eq:explicit-comparison2} {is strictly negative. Direct differentiation gives}
	\begin{equation*}
		(K_p^{-3}\Phi_p)_p=-2C_2K_p^3<0
	\end{equation*}
	and
	\begin{equation}\label{eq:explicit-comparison2}
		-K_p^{-3}\Phi_p+g\Phi\bigg|_{\mathcal{T}}=g-C_1\left(1-\frac{1}{F^2}\right)+C_2\left(\frac{2}{gF^2}-g\left(\frac{1}{gF^2}\right)^2\right).
	\end{equation}
	The assumption $C_1>\frac{2 C_2}{gF^2}$ implies $\Phi_p(p)>0$ and hence $\Phi\geq1$.

	\subsection{{Limiting invertibility and the Fredholm property}}
	{The upstream and downstream profiles generally give different limiting operators. We first prove that both are invertible on bounded H\"older spaces, then recover the endpoint limits of solutions, and finally compute the Fredholm index by two continuous homotopies to a constant-coefficient operator.}

	\begin{lemma}\label{lem:limit-invertibility}
	If $\sigma_0^{\pm}<0$, then
	$\mathscr{L}_{\pm}:\mathscr{X}_b\to \mathscr{Y}_b$
	is an isomorphism.
	\end{lemma}
	\begin{proof}
	{We give the proof for the downstream limiting operator. Write $\mathscr L_{+,1}$ and $\mathscr L_{+,2}$ for its interior and boundary components, and let $\Phi_+:=\Phi[K_+]$ be the positive comparison function in \eqref{eq:explicit-comparison}. Consider the shifted family}
\begin{equation*}
	\mathscr L_{+,t}:=
	(\mathscr L_{+,1}-t\tau,\mathscr L_{+,2}),
	\qquad 0\le t\le1,
	\end{equation*}
	{where $\tau>0$ is fixed. Denote by $\sigma_0(t)$ the principal transverse eigenvalue of $\mathscr L_{+,t}$; it is given by the variational formula}
	\begin{equation*}
		\sigma_0(t)=\sup_{\substack{\varphi\in H^1(-m,0)\\
				\varphi(-m)=0,\ \varphi\not\equiv0}}
		\frac{g\varphi(0)^2-
			\displaystyle\int_{-m}^0K_{p,+}^{-3}\varphi_p^2d p-
			t\tau\displaystyle\int_{-m}^0\varphi^2d p}
		{\displaystyle\int_{-m}^0K_{p,+}^{-1}\varphi^2d p}.
	\end{equation*}
	{The variational formula gives $\sigma_0(t)\leq\sigma_0^+<0$ for $t\in[0,1]$. To justify injectivity on the bounded H\"older space, let $u\in\mathscr X_b$ solve $\mathscr L_{+,t}u=0$ and divide by the positive strict supersolution $\Phi_+$. If the quotient is nonzero, translate a sequence approaching its positive supremum (or negative infimum) in the $q$-direction. Boundary Schauder compactness gives a limiting bounded solution whose quotient attains a nonzero extremum. The strong maximum principle, together with the boundary point lemma on $\mathcal T$ and the Dirichlet condition on $\mathcal B$, gives a contradiction. Hence $u=0$, uniformly for $t\in[0,1]$.}
	
	{We next prove the a priori estimate}
	\begin{equation}\label{eq:limit-apriori}
		\|\phi\|_{C_b^{3+\alpha}}\leq C\|{\mathscr L_{+,t}}\phi\|_{\mathscr{Y}_b},\qquad t\in[0,1].
	\end{equation}
	{The Agmon--Douglis--Nirenberg Schauder estimate~\cite{agmon} gives}
	\begin{equation}\label{eq:limit-global-Schauder}
	\|\phi\|_{C_b^{3+\alpha}}\leq C(\|\mathscr{L}_{+,t}\phi\|_{\mathscr{Y}_b}+\|\phi\|_{C^0}),
	\end{equation}
	{where $C$ is independent of $t$. It remains to bound the $C^0$ term by the norm of the data. If such a bound failed, there would be sequences $t_n\to t_*\in[0,1]$ and $\phi_n\in\mathscr X_b$ such that}
	\begin{equation*}
		\|\phi_n\|_{C^0}=1\qquad \|\mathscr{L}_{+,t_n}\phi_n\|_{\mathscr{Y}_b}\to0.
	\end{equation*}
	{Estimate~\eqref{eq:limit-global-Schauder} bounds $\phi_n$ in $C_b^{3+\alpha}$.}
	 {Choose points $(q_n,p_n)$ such that $|\phi_n(q_n,p_n)|\geq1/2$, and translate by setting $\tilde\phi_n(q,p)=\phi_n(q+q_n,p)$. After passing to a subsequence, boundary Schauder compactness gives a limit $\phi_*$ in $C_{\mathrm{loc}}^3(\overline\Omega)$ satisfying $\mathscr L_{+,t_*}\phi_*=0$. We may also assume $p_n\to p_*$. The uniform gradient bound and the zero boundary values on $\mathcal B$ imply $p_*>-m$, while $|\phi_*(0,p_*)|\geq1/2$. This contradicts injectivity at $t=t_*$.}

 {To prove surjectivity, fix data $(f_1,f_2)\in\mathscr Y_b$. For $R>0$, define the bounded rectangle}
 \begin{equation*}
 	{\Omega_R:=(-R,R)\times(-m,0)},
 \end{equation*}
 and the function space
 \begin{equation*}
 	V_R:=\{u\in H^1(\Omega_R):u=0\mathrm{~on~}\mathcal B,~u=0\mathrm{~on~}|q|=R\},
 \end{equation*}
 {For $u,v\in V_R$, consider the bilinear form}
 \begin{equation*}
 	 \begin{split}
 		\mathfrak a_{R}[u,v]
 		:= {}&\int_{\Omega_R}
		\bigl(K_{p,+}^{-1}u_qv_q+K_{p,+}^{-3}u_pv_p+t\tau uv\bigr)\,d qd p-g\int_{-R}^{R}u(q,0)v(q,0)\,d q .
 	\end{split}
 \end{equation*}
 {The weighted trace estimate \eqref{eq:transverse-coercivity} and the Poincar\'e inequality give}
 \begin{equation*}
 	\begin{split}
 		\mathfrak a_{R}[u,u]
 		&\ge \int_{\Omega_R}K_{p,+}^{-1}u_q^2\,d qd p
 		+(1-F(K_+)^{-2})
 		\int_{\Omega_R}K_{p,+}^{-3}u_p^2\,d qd p\\
 		&\ge c\|u\|_{H^1(\Omega_R)}^2,
 	\end{split}
 \end{equation*}
 where $c>0$ is independent of $R$ and $t$. {The Lax--Milgram theorem gives a unique weak solution $u_R\in V_R$ of the inhomogeneous problem}
 \begin{equation*}
 	\begin{cases}
		({\mathscr L_{+,1}}-t\tau)u_R=f_1&\text{in }\Omega_R,\\
		{\mathscr L_{+,2}}u_R=f_2&\text{on }\mathcal T\cap\{q\in(-R,R)\},\\
		u_R=0&\text{on }\mathcal B\cap\{q\in(-R,R)\},\\
		{u_R=0}&{\text{on }\{q=\pm R\}\times[-m,0].}
	\end{cases}
 \end{equation*}

 {To pass to the infinite strip, we derive an estimate independent of $R$. Let}
 \begin{equation*}
 	\Phi_+(p):=\Phi[K_+](p)=1+C_1\int_{-m}^pK_{p,+}^3(t)dt-C_2\left(\int_{-m}^pK_{p,+}^3(t)\right)^2
 \end{equation*}
 {be the comparison function from \eqref{eq:explicit-comparison}. There is a constant $c_*>0$, independent of $t$, such that}
 \begin{equation*}
 	 ({\mathscr L_{+,1}}-t\tau)\Phi_+\leq-c_*,
 	\qquad {\mathscr L_{+,2}}\Phi_+\leq-c_*,
 	\qquad 0\leq t\leq1.
 \end{equation*}
 {Choose $M_*>0$ large enough that}
 \begin{equation*}
 	M_*c_*\geq \|f_1\|_{C^0(\overline\Omega)}+\|f_2\|_{C^0(\mathcal T)}.
 \end{equation*}
 {Define the upper and lower comparison functions $W_R^\pm=M_*\Phi_+\pm u_R$. Then}
 \begin{equation*}
	(\mathscr{L}_{+,1}-t\tau)W_{R}^{\pm}={M_*(\mathscr L_{+,1}-t\tau)\Phi_+\pm(\mathscr L_{+,1}-t\tau)u_R}\leq -M_* c_* \pm f_1 \leq 0,
 \end{equation*}
 and
  \begin{equation*}
 	\mathscr{L}_{+,2}W_{R}^{\pm}\leq 0.
 \end{equation*}
	 Moreover, on the vertical boundary $\{q=\pm R\}$, we have $W^\pm_{R}={M_*\Phi_+}>0$.

 {Writing $W_R^\pm=\Phi_+V_R^\pm$ defines the quotients $V_R^\pm$. They satisfy the differential inequalities associated with \eqref{eq:quotient-equation}, with the strictly negative interior zeroth-order coefficient $\mathscr L_{+,1}\Phi_+/\Phi_+-t\tau$.} {The quotient maximum-principle argument used in Lemma \ref{lem:comparison-function}, now with nonnegative data on $\mathcal B$ and on the vertical sides,} implies $V_R^{\pm}\geq0$. Hence
\begin{equation}\label{eq:finite-cylinder-C0}
	\|u_R\|_{C^0(\overline\Omega_R)}
	\le M_*\|\Phi_+\|_{C^0([-m,0])},
\end{equation}
 uniformly in $R$ and $t$.

 {Schauder estimates~\cite{Ts} on unit strips at distance at least one from the vertical sides, together with \eqref{eq:finite-cylinder-C0}, yield}
 \begin{equation}\label{eq:finite-cylinder-local-Schauder}
 	\|u_R\|_{C^{3+\alpha}((j,j+1)\times[-m,0])}
 	\le C\bigl(\|f_1\|_{C_b^{1+\alpha}}
 	+\|f_2\|_{C_b^{2+\alpha}}\bigr)
 \end{equation}
	 whenever \(j\in\R\) and \((j-1,j+2)\subset(-R,R)\), with \(C\) independent of
	 \(j,R,t\).  Letting \(R\to\infty\), a diagonal compactness argument produces
 \(u\in C_b^{3+\alpha}(\overline\Omega)\), zero on \(\mathcal B\), which solves
 \({\mathscr L_{+,t}}u=(f_1,f_2)\).  Estimate
 \eqref{eq:finite-cylinder-local-Schauder} passes to the limit, so
 \(u\in\mathscr X_b\).  Thus every \({\mathscr L_{+,t}}\) is onto.  Injectivity and
 \eqref{eq:limit-apriori} make it an isomorphism and give the uniform inverse
 bound.    {The same argument applies to the upstream limiting operator $\mathscr L_-$.}
\end{proof}

	\begin{lemma}\label{lem:range-regularity}
	{Assume $\sigma_0^\pm<0$, and let $f=(f_1,f_2)\in\mathscr Y_\infty$ denote the interior and boundary data. If $\phi\in\mathscr X_b$ satisfies $\mathscr L\phi=f$, then $\phi\in\mathscr X_\infty$. Consequently,}
	\begin{equation}\label{eq:range-intersection}
		\mathscr L(\mathscr X_\infty)=\mathscr L(\mathscr X_b)\cap\mathscr Y_\infty,
		\qquad
		\ker(\mathscr L|_{\mathscr X_b})
		=\ker(\mathscr L|_{\mathscr X_\infty}).
		\end{equation}
	\end{lemma}
	\begin{proof}
		{We first prove convergence at $+\infty$ for $\phi$ and its derivatives of order at most three.} We claim that for any multi-index $\nu=(\nu_q,\nu_p)$ with $|\nu|\leq3$, the family $\{D^\nu \phi(q,\cdot):q\geq Q\}$ is Cauchy in $C^0([-m,0])$ as $Q\to\infty$. If not, there exist some $\epsilon>0$, sequences $q_{1n},q_{2n}\to\infty$ and $p_n\to p_*\in[-m,0]$ such that
		\begin{equation}\label{eq:endpoint-separation}
		|D^\nu\phi(q_{1n},p_n)-D^\nu\phi(q_{2n},p_n)|\geq\epsilon.
		\end{equation}
		{Define the two translated functions $\phi_n^{(k)}$ and their difference $v_n$ by}
		\begin{equation*}
			\phi_n^{(k)}=\phi(q+q_{kn},p),\qquad v_n=\phi_n^{(1)}-\phi_n^{(2)},\quad k=1,2,
		\end{equation*}
		{Let $\mathscr L_n^{(k)}$ be the operator obtained by translating the coefficients of $\mathscr L$ by $q_{kn}$. Subtracting the two translated equations, including their boundary components, gives}
		\begin{equation}\label{eq:translated-difference-equation}
			{\mathscr L_n^{(1)}}v_n
			=f(\,\cdot+q_{1n},\cdot)-f(\,\cdot+q_{2n},\cdot)
			+({\mathscr L_n^{(2)}}-{\mathscr L_n^{(1)}})\phi_n^{(2)}.
		\end{equation}
		Because \(h\in H+\mathscr X_\infty\), the coefficients of both
		\({\mathscr L_n^{(i)}}\), together with the derivatives needed to pass to the
		divergence equation and oblique boundary condition, converge uniformly on
		compact subsets to those of \(\mathscr L_+\).  Since \(f\in\mathscr Y_\infty\), the
		first difference on the right of
		\eqref{eq:translated-difference-equation} tends to zero locally, together
		with the corresponding interior and boundary derivatives.  The second
		difference also tends to zero locally: the coefficient difference converges
		to zero, while \(\phi_n^{(2)}\) is uniformly bounded in
		\(C_b^{3+\alpha}\).
		
		The sequence \(v_n\) is bounded in \(C_b^{3+\alpha}\).  Local compact 	embedding, including the Dirichlet and oblique boundaries, therefore gives a	subsequence such that
		\begin{equation*}
			v_n\longrightarrow v\quad\hbox{in }C_{\rm loc}^{3}(\overline\Omega),
			\qquad {\mathscr L_+}v=0,\qquad v=0\ \hbox{on }\mathcal B .
		\end{equation*}
		The uniform H\"older bounds pass to the limit, so \(v\in\mathscr X_b\); in particular, it is bounded on the whole strip.  Lemma~\ref{lem:limit-invertibility}
		forces \(v=0\), whereas \eqref{eq:endpoint-separation} and the
		\(C^3_{\rm loc}\) convergence give \(|D^\nu v(0,p_*)|\ge\epsilon\).
		This contradiction proves our claim for every
		\(|\nu|\le3\).
		
		Let \(\ell_\nu(p)\) be the uniform limit of $D^\nu\phi(q,p)$.  If \(\nu_q\ge1\), then
		\(D^{\nu-(1,0)}\phi\) is bounded and its \(q\)-derivative tends uniformly to
		\(\ell_\nu\).  A nonzero value of \(\ell_\nu(p)\) would make
		\(D^{\nu-(1,0)}\phi(q,p)\) unbounded, after fixing
		\(p\) and integrating with respect to \(q\).   Hence
		\begin{equation}\label{eq:tangential-limits-zero}
			\ell_\nu\equiv0\qquad\text{whenever }\nu_q\ge1,\quad |\nu|\le3.
		\end{equation}
		Define \(\phi_+(p):=\ell_{(0,0)}(p)\).  Uniform convergence of the
		{$p$-derivatives and the fundamental theorem of calculus show that}
		\[
		\phi_+\in C^3([-m,0]),\qquad
		\partial_p^j\phi_+=\ell_{(0,j)},\quad 0\le j\le3.
		\]
		The uniform \(C_b^{3+\alpha}\) bound for \(\phi\) passes to the limit in
		the H\"older seminorm, so in fact \(\phi_+\in C^{3+\alpha}([-m,0])\).
	 {Let $f_{1,+}$ and $f_{2,+}$ denote the downstream limits of the interior and boundary data. Passing to the limit in the equations yields}
		\begin{equation*}
			(K_{p,+}^{-3}\phi_{+,p})_p=f_{1,+},\qquad
			-K_{p,+}(0)^{-3}\phi_{+,p}(0)+g\phi_+(0)=f_{2,+},\qquad
			\phi_+(-m)=0 .
		\end{equation*}
	{The Cauchy property proved above and \eqref{eq:tangential-limits-zero} give}
		\begin{equation}\label{eq:full-endpoint-convergence}
			\max_{i+j\le3}\sup_{-m\le p\le0}
			\left|\partial_q^i\partial_p^j
			\bigl(\phi(q,p)-\phi_+(p)\bigr)\right|
			\longrightarrow0\qquad(q\to+\infty).
		\end{equation}
		
		The same argument at \(-\infty\) produces
		\(\phi_-\in C^{3+\alpha}([-m,0])\) and the analogue of
		\eqref{eq:full-endpoint-convergence}. With the fixed cutoffs from
		{\eqref{eq:two-ended-space}, define the decaying remainder $\phi_0$ by}
		\[
		\phi_0:=\phi-\chi_-\phi_- -\chi_+\phi_+.
		\]
		Then \(\phi_0\in C_b^{3+\alpha}(\overline\Omega)\), and every derivative of
		\(\phi_0\) of order at most three tends uniformly to zero at both ends.
		Thus \(\phi_0\in C_b^{3+\alpha}\cap C_0^3\), proving
		\(\phi\in\mathscr X_\infty\).
		Finally,
		\(\mathscr L(\mathscr X_\infty)\subset\mathscr L(\mathscr X_b)\cap\mathscr Y_\infty\) is
		immediate.  Conversely, if an element of the intersection equals
		\(\mathscr L\phi\) for some \(\phi\in\mathscr X_b\), the result just proved puts
		\(\phi\) in \(\mathscr X_\infty\).  Taking the right-hand side to be zero gives the
		kernel identity in \eqref{eq:range-intersection}.	
	\end{proof}
	
	  \begin{lemma}[\cite{ChenWalshWheelerFronts,vol}]\label{lem:citelemmaa3}
	  {The elliptic operator $\mathscr L:\mathscr X_b\to\mathscr Y_b$ is locally proper if and only if both limiting operators $\mathscr L_\pm:\mathscr X_b\to\mathscr Y_b$ are injective.}
	  \end{lemma}

	\begin{prop}[Fredholm property]\label{prop:fredholm}
	For $\sigma_0^\pm<0$, the operator
		\[
		\mathscr L:\mathscr X_\infty\longrightarrow\mathscr Y_\infty
		\]
		is Fredholm of index zero.
	\end{prop}
	\begin{proof}
		{Since $h=H+w$ with $w\in\mathscr X_\infty$, the coefficients of $\mathscr L$ converge at the two ends to those of $\mathscr L_\pm$. Lemma~\ref{lem:limit-invertibility} makes both limiting operators isomorphisms on the bounded H\"older spaces. Lemma~\ref{lem:citelemmaa3} therefore implies that}
		\begin{equation*}
			\mathscr{L}:\mathscr{X}_b\to\mathscr{Y}_b
		\end{equation*}
		{is locally proper, or equivalently upper semi-Fredholm.}

		 {We first compute the index on $\mathscr X_b$. Let $\tau>0$ be a constant to be chosen below, and define the first homotopy by}
	\begin{equation*}
		\mathcal H_t:=({\mathscr L_1}-t\tau,{\mathscr L_2}),
		\qquad0\le t\le1 ,
	\end{equation*}
	{whose limiting transverse principal eigenvalues, denoted by $\sigma_{0,t}^\pm$, satisfy}
\begin{equation*}
	\sigma_{0,t}^\pm=
	\sup_{\substack{\varphi\in H^1(-m,0)\\
			\varphi(-m)=0,\ \varphi\not\equiv0}}
	\frac{g\varphi(0)^2-
		\displaystyle\int_{-m}^0K_{p,\pm}^{-3}\varphi_p^2d p-
		t\tau\displaystyle\int_{-m}^0\varphi^2d p}
	{\displaystyle\int_{-m}^0K_{p,\pm}^{-1}\varphi^2d p}
	\le\sigma_0^\pm<0.
\end{equation*}
	{Let $\mathcal H_{t,\pm}$ denote the limiting operators of $\mathcal H_t$. The proof of Lemma~\ref{lem:limit-invertibility} shows that they are isomorphisms from $\mathscr X_b$ to $\mathscr Y_b$, with inverse bounds uniform in $t$. Lemma~\ref{lem:citelemmaa3} then implies that}
	\begin{equation*}
		\mathcal{H}_t:\mathscr{X}_b\to\mathscr{Y}_b
	\end{equation*}
	is upper semi-Fredholm.
	
	{Next, with $\partial_{\boldsymbol n}$ denoting outward normal differentiation, define a second homotopy, beginning at $\mathcal H_1$, by}
	\begin{equation*}
		\mathcal M_t:=
		\bigl(t\Delta+(1-t){\mathscr L_1}-\tau,\,
		-t\partial_{\boldsymbol n}+(1-t){\mathscr L_2}\bigr),
		\qquad0\le t\le1 .
	\end{equation*}
	{The operators remain uniformly elliptic and uniformly oblique. Denote their limiting transverse principal eigenvalues by $\widehat\sigma_{0,t}^\pm$. For each sign, let $b_t^\pm$ and $c_t^\pm$ denote the convex combinations of the transverse and longitudinal leading coefficients, respectively, defined below. Then the eigenvalues are given by the Rayleigh quotient}
	\begin{equation}\label{eq:second-homotopy-spectrum}
	          	\widehat\sigma_{0,t}^\pm=
	          	\sup_{\substack{\varphi\in H^1(-m,0)\\
	          			\varphi(-m)=0,\ \varphi\not\equiv0}}
	          	\frac{(1-t)g\varphi(0)^2-
	          		\displaystyle\int_{-m}^0b_t^\pm\varphi_p^2d p-
	          		\tau\displaystyle\int_{-m}^0\varphi^2d p}
	          	{\displaystyle\int_{-m}^0c_t^\pm\varphi^2d p},
	\end{equation}
	where \begin{equation*}
		b_t^{\pm}:=t+(1-t)K_{p,\pm}^{-3},\qquad c_t^{\pm}:=t+(1-t)K_{p,\pm}^{-1}.
	\end{equation*}
	 {For $\tau>0$ sufficiently large, the variational formula \eqref{eq:second-homotopy-spectrum} gives, for some $\kappa>0$, the uniform bound}
	\begin{equation*}
		\hat{\sigma}_{0,t}^{\pm} \leq -\kappa < 0\quad\mathrm{for~all~}t\in[0,1].
	\end{equation*}
	  	{Let $\mathcal M_{t,\pm}$ denote the two limiting operators of $\mathcal M_t$.}
	{The proof of Lemma \ref{lem:limit-invertibility} applies verbatim to $\mathcal M_{t,\pm}$ after replacing $K_{p,\pm}^{-3}$, $K_{p,\pm}^{-1}$, and $g$ by $b_t^\pm$, $c_t^\pm$, and $(1-t)g$, respectively. Indeed, these coefficients are uniformly elliptic, the displayed spectral bound supplies the same positive strict supersolution and injectivity argument, and the finite-cylinder Lax--Milgram and Schauder construction supplies surjectivity with bounds uniform in $t$. Thus $\mathcal M_{t,\pm}:\mathscr X_b\to\mathscr Y_b$ is an isomorphism uniformly in $t$.} Moreover, Lemma \ref{lem:citelemmaa3} implies
	  \begin{equation*}
	  	\mathcal{M}_t:\mathscr{X}_b\to\mathscr{Y}_b
	  \end{equation*}
	  is upper semi-Fredholm.
	
	At $t=1$, we have
	\begin{equation*}
		\mathcal{M}_1\phi=(\Delta\phi-\tau\phi,-\partial_\mathbf{n}\phi)
	\end{equation*}
	with zero Dirichlet condition on $\mathcal{B}$.  Its limiting transverse principal eigenvalue is strictly negative, so Lemma    \ref{lem:limit-invertibility}      makes $\mathcal{M}_1:\mathscr{X}_b\to\mathscr{Y}_b$
	 an isomorphism. By homotopy invariance of the extended index on the upper semi-Fredholm operators, we can deduce that
	\begin{equation*}
		\mathrm{ind}(\mathscr{L}|_{\mathscr{X}_b})=\mathrm{ind} (\mathcal{M}_1)=0.
	\end{equation*}
	In particular, {all operators along the two homotopies are Fredholm.}
	
	{We now restrict the operators to $\mathscr X_\infty$. Let $A$ denote any operator on either homotopy path. Its coefficients converge at both ends, and its limiting operators are invertible. The argument of Lemma~\ref{lem:range-regularity} therefore gives}
	\begin{equation*}
		A(\mathscr{X}_\infty)=A(\mathscr{X}_b)\cap\mathscr{Y}_\infty,\qquad\mathrm{ker}(A|_{\mathscr{X}_\infty})=\mathrm{ker}(A|_{\mathscr{X}_b}).
	\end{equation*}
	{Since $A:\mathscr X_b\to\mathscr Y_b$ is Fredholm, its range is closed and its kernel is finite-dimensional. As $\mathscr Y_\infty$ is a closed subspace of $\mathscr Y_b$, these identities imply that}
	\begin{equation*}
		A:\mathscr{X}_\infty\to\mathscr{Y}_\infty
	\end{equation*}
	{is upper semi-Fredholm. Finally, $\mathcal M_1$ is onto on the bounded spaces, and Lemma~\ref{lem:range-regularity} shows that the solution of $\mathcal M_1\phi=f$ belongs to $\mathscr X_\infty$ whenever $f\in\mathscr Y_\infty$. Thus $\mathcal M_1$ is also an isomorphism on the two-ended spaces. Index invariance now gives}
	\begin{equation*}
		\mathrm{ind}(\mathscr{L}|_{\mathscr{X}_\infty})=\mathrm{ind}(\mathcal{M}_1|_{\mathscr{X}_\infty})=0,
	\end{equation*}
	which completes the proof.
	\end{proof}

	\subsection{{Small-amplitude fronts}}
	{In particular, the laminar linearization $\mathcal L:=\mathcal F_w(0,0)$ is invertible on the two-ended spaces.}
	\begin{prop}\label{prop:f0inverse}
		For the supercritical flow,
		$\mathcal{L}:\mathscr{X}_\infty\to\mathscr{Y}_\infty$
		is an isomorphism.
	\end{prop}
	{
	\begin{proof}
	At $(\beta,w)=(0,0)$, the two limiting operators coincide with $\mathcal L$. Lemma \ref{lem:limit-invertibility} therefore makes $\mathcal L:\mathscr X_b\to\mathscr Y_b$ an isomorphism. If $f\in\mathscr Y_\infty$, Lemma \ref{lem:range-regularity} shows that the unique solution in $\mathscr X_b$ belongs to $\mathscr X_\infty$. Hence $\mathcal L:\mathscr X_\infty\to\mathscr Y_\infty$ is an isomorphism.
	\end{proof}}
	
	{We can now construct the local curve of fronts in $\R\times\mathscr X_\infty$.}
\begin{theorem}[Small-amplitude pressure fronts]\label{thm:local}
	{Suppose that $\gamma\in C^{4+\alpha}([0,m])$, the upstream Froude number satisfies $F>1$, and Assumption~\ref{ass:pressure} holds.}  There are \(\varepsilon>0\), a neighborhood
	\(\mathcal V\subset\mathscr X_\infty\) of~ \(w=0\), and a unique real-analytic mapping
	\[
	(-\varepsilon,\varepsilon)\ni\beta\longmapsto w(\beta)\in\mathcal V
	\]
	{satisfying}
	\[
	\mathcal F(\beta,w(\beta))=0,\qquad w(0)=0.
	\]
	The height \(h(\beta)=H+w(\beta)\) has the limits
	\[
	h(\beta)(q,\cdot)\longrightarrow
	\begin{cases}
		H(\cdot;\lambda),&q\to-\infty,\\
		H(\cdot;\lambda_+(\beta)),&q\to+\infty,
	\end{cases}
	\]
	in \(C^3([-m,0])\).  For \(0<|\beta|<\varepsilon\),
	\begin{equation*}
		\mathrm{sgn}(\beta)\,h_q(\beta)>0
		\quad\text{in }\Omega\cup\mathcal T,
		\qquad
		\mathrm{sgn}(\beta)\,\eta_q(\beta)>0.
	\end{equation*}
	Thus \(\beta>0\) gives a front of elevation and \(\beta<0\) a front of
	depression. Moreover, their depth jump satisfies
	\begin{equation}\label{eq:local-depth-expansion}
		d_+(\beta)-d
		=\frac{\beta}{g(F^2-1)}+O(\beta^2).
	\end{equation}
\end{theorem}
	\begin{proof}
		{Proposition~\ref{prop:f0inverse} and the analytic implicit function theorem~\cite[Theorem 4.5.4]{BuffoniToland} give the local solution curve. Its endpoint profiles follow from Lemma~\ref{lemma:analytic-map}, and \eqref{eq:depth-derivative} gives the depth expansion \eqref{eq:local-depth-expansion}.}
		
		{To prove monotonicity, let $z=h_q$ be the horizontal derivative of the height function, and let $A(h)$ denote the principal coefficient matrix of the linearized interior equation. Differentiating \eqref{heighteq} with respect to $q$ gives}
		\begin{align*}
			\operatorname{div}(A(h)\nabla z)&=0&&\text{in }\Omega,\\
			-\mathbf n\cdot A(h)\nabla z+gz&=-\beta r'(q)&&\text{on }\mathcal T,\\
			z&=0&&\text{on }\mathcal B,
		\end{align*}
	where \begin{equation}\label{eq:principal-matrix}
		 \mathbb M_2(C_b^{2+\alpha}(\overline\Omega))\ni A(h):=
		\begin{pmatrix}
			h_p^{-1}&-h_qh_p^{-2}\\
			-h_qh_p^{-2}&(1+h_q^2)h_p^{-3}
		\end{pmatrix}.
	\end{equation}
	{Let $\Phi_H$ denote the comparison function for $H$:}
	\begin{equation*}
	\Phi_H:=\Phi[H](p)=1+C_1\int_{-m}^pH_p^3(t)dt-C_2\left(\int_{-m}^pH_p^3(t)\right)^2
	\end{equation*}
	{By \eqref{eq:explicit-comparison}, after reducing $\varepsilon>0$ if necessary, continuity gives the strict inequalities}
	\[
	\operatorname{div}(A(h(\beta))\nabla\Phi_H)<0,
	\qquad
	-\mathbf n\cdot A(h(\beta))\nabla\Phi_H+g\Phi_H<0.
	\]
{For $\beta>0$, Lemma~\ref{lem:comparison-function}, applied with $u=z$ and $f=-\beta r'<0$, yields $h_q>0$ in $\Omega\cup\mathcal T$. For $\beta<0$, applying the same lemma to $u=-z$ yields $h_q<0$ in $\Omega\cup\mathcal T$.}
	\end{proof}

	\section{{Global monotonicity}}\label{mono}
		{We now show that the monotonicity obtained locally persists under continuation. For a solution $w$ of \eqref{wpde}, the \emph{depression properties} are}
	\begin{equation}\label{nodaldep}
		w_q<0\quad\mathrm{in~}\Omega\cup \mathcal T,
	\end{equation}
	and \begin{equation*}
		w_{qp}<0\quad\mathrm{on~}\mathcal B.
	\end{equation*}
	
	{Similarly, the \emph{elevation properties} are}
	\begin{equation}\label{nodalele}
		w_q>0\quad\mathrm{in~}\Omega\cup \mathcal T,
	\end{equation}
	and
	\begin{equation*}
		w_{qp}>0\quad\mathrm{on~}\mathcal B.
	\end{equation*}
	{Because $w_q$ tends to zero at both ends of the strip, strict monotonicity alone does not immediately imply openness. The next lemma combines a bounded-region argument with comparison functions on the two tails.}
	
	\begin{lemma}[Open condition]\label{lem:open}
		{In the supercritical setting above, fix a solution $(\beta^*,w^*)\in\mathcal F^{-1}(0)$ and write $h^*=H+w^*$.}
	
		(a) If \(\beta^* > 0\) and
		\[
		h_q^* > 0 \quad \text{in } \Omega \cup \mathcal{T},
		\]
		then there exists \(\epsilon > 0\) such that every solution {$(\beta,w)\in\mathcal F^{-1}(0)$} satisfying
		\[
		\|w - w^*\|_{C_b^3(\bar{\Omega})} + |\beta - \beta^*| < \epsilon, \quad \mathrm{with~}\beta > 0,
		\]
		also satisfies
		\[
		h_q > 0 \quad \text{in } \Omega \cup \mathcal{T}.
		\]
		
		(b) If \(\beta^* < 0\) and
		\[
		h_q^* < 0 \quad \text{in } \Omega \cup \mathcal{T},
		\]
		then there exists \(\epsilon > 0\) such that every solution {$(\beta,w)\in\mathcal F^{-1}(0)$} satisfying
		\[
		\|w - w^*\|_{C_b^3(\bar{\Omega})} + |\beta - \beta^*| < \epsilon, \quad\mathrm{with~} \beta < 0,
		\]
		also satisfies
		\[
		h_q < 0 \quad \text{in } \Omega \cup \mathcal{T}.
		\]
	\end{lemma}
	\begin{proof}
		The proofs of $(a)$ and $(b)$ are identical after replacing $h_q$ by $-h_q$, {so it suffices to prove (a).}
		{Let $K_\pm^*$ be the limiting laminar profiles of $h^*$, and let $\Phi_\pm^*:=\Phi[K_\pm^*]$ be their comparison functions from \eqref{eq:explicit-comparison}. The coefficients of the linearized operator at $h^*$ converge in $C^1$ to those of the corresponding limiting operators at the two ends. Consequently, there are constants $K,c,\epsilon_1>0$ such that every solution satisfying}
		\begin{equation*}
				\|w - w^*\|_{C_b^3(\bar{\Omega})} + |\beta - \beta^*| < \epsilon_1,
		\end{equation*}
		we have
		\begin{equation*}
			\operatorname{div}(A(h)\nabla\Phi_{\pm}^*) \leq -\frac{c}{2}
		\end{equation*}
		and
		\begin{equation*}
			-\mathbf{n} \cdot A(h) \nabla \Phi_{\pm}^* + g \Phi_{\pm}^* \leq -\frac{c}{2},
		\end{equation*}
		{on the corresponding tail $\pm q>K$.}
		
		We split the domain $\Omega$ into overlapping regions
		\begin{equation*}
			\Omega_1:=\{(q,p)\in\Omega:|q|<2K\},\qquad\Omega_2:=\{(q,p)\in\Omega:|q|>K\},
		\end{equation*}
		{Denote their top and bottom boundary portions by $\mathcal T_i$ and $\mathcal B_i$, respectively, for $i=1,2$. We first consider the bounded region $\Omega_1$, where}
		\begin{equation*}
			h_q^* > 0 \quad {\text{in }\Omega_1\cup\mathcal T_1}, \quad h_q^* = 0 \quad {\text{on }\mathcal B_1}.
		\end{equation*}
		{Hopf's lemma therefore gives}
		\begin{equation*}
			{h_{qp}^*(q,-m)>0,\qquad|q|<2K.}
		\end{equation*}
		 {By continuity, there is $\epsilon_2>0$ such that every solution satisfying}
		\begin{equation*}
				\|w - w^*\|_{C_b^3(\bar{\Omega})} + |\beta - \beta^*| < \epsilon_2,
		\end{equation*}
		{has $h_q>0$ in $\Omega_1\cup\mathcal T_1$.}
		In particular,
		\[
		h_q \geq 0 \quad \text{on } \{q = \pm K\}.
		\]
		
		{It remains to consider the two tails of $\Omega_2$. Set $z=h_q$ and define the quotient $v$ on each tail by}
		\[ z = \Phi_{\pm}^* v \quad \text{in } \Omega_2. \]
		{Differentiating \eqref{heighteq} with respect to $q$,}
		\[ \operatorname{div}(A(h)\nabla z) = 0 \quad \text{in } \Omega_2 \]
		and
		\[ {-\mathbf n\cdot} A(h)\nabla z + gz = -\beta r' < 0 \quad \text{on } \mathcal T_2. \]
		Consequently,
		\[
		\operatorname{div}(A(h)\nabla v) + 2A(h)\nabla(\log \Phi_{\pm}^*) \cdot \nabla v
		+ \frac{\operatorname{div}(A(h)\nabla \Phi_{\pm}^*)}{\Phi_{\pm}^*}v = 0
		\]
		in \(\Omega_2\), and
		\[
		{-\mathbf n\cdot} A(h)\nabla v
		+ \frac{{-\mathbf n\cdot} A(h)\nabla \Phi_{\pm}^* + g \Phi_{\pm}^*}{ \Phi_{\pm}^*}v
		= -\frac{\beta r'}{ \Phi_{\pm}^*} < 0
		\]
		on \(\mathcal T_2\).
		Moreover,
		\[
		v = 0 \quad {\text{on }\mathcal B_2}, \quad v \geq 0 \quad \text{on } \{q = \pm K\},
		\]
		and \(v \to 0\) at the corresponding infinite end. The comparison function argument in Lemma \ref{lem:comparison-function} therefore gives
		\[
		h_q = z \geq 0 \quad \text{in } \Omega_2.
		\]
		{The strong maximum principle and Hopf's lemma improve this to}
		\[
		h_q > 0 \quad \text{in } \Omega_2 \cup \mathcal T_2,
		\]
		due to
		\[
		{-\mathbf n\cdot} A(h)\nabla h_q + gh_q = -\beta r' < 0.
		\]
		{Combining the estimates on the bounded region and the tails gives}
		\[
		h_q > 0 \quad \text{in } \Omega \cup \mathcal T
		\]
		by taking
		\[
		\epsilon := \min\{\epsilon_1, \epsilon_2, \beta^* / 2\},
		\]
		{and Hopf's lemma on $\mathcal B$ also gives} $h_{qp}>0$ on $\mathcal B$,
		 which completes the proof.
		\end{proof}

		\begin{lemma}[Closed condition]\label{lem:closed}
			{Let $(\beta_n,w_n)$ be a sequence of solutions such that}
			\[
			( \beta_n,w_n) \to (\beta,w) \quad \text{in }\R\times\mathscr X_{\infty},
			\]
			where \((\beta,w)\) is also a solution. Let
			\[
			h_n = H + w_n, \quad h = H + w.
			\]
			
			(a) Suppose that
			\[
			\beta_n > 0, \quad h_{nq} > 0 \quad \mathrm{in~} \Omega \cup \mathcal T, \quad \beta \geq 0.
			\]
			Then either
			\[
			{(\beta,w)=(0,0)},
			\]
			or else \(\beta > 0\) and
			\[
			h_q > 0 \quad \mathrm{in~} \Omega \cup\mathcal T.
			\]
			
			(b) Suppose that
			\[
			\beta_n < 0, \quad h_{nq} < 0 \quad \mathrm{in~} \Omega \cup\mathcal T, \quad \beta \leq 0.
			\]
			Then either
			\[
			{(\beta,w)=(0,0)},
			\]
			or else \(\beta < 0\) and
			\[
			h_q < 0 \quad \mathrm{in~} \Omega \cup\mathcal T.
			\]
		\end{lemma}
		\begin{proof}
			The proofs of (a) and (b) are the same, so we prove only (a). The convergence in \( \mathscr{X}_{\infty} \) implies
			\[
			h_{nq} \to h_q
			\]
			uniformly in \( \overline{\Omega} \). Hence
			$h_q \geq 0 \quad \text{on } \overline{\Omega}.$
			
			First suppose that \( \beta > 0 \). Setting \( z = h_q \), we have
			\[
			\operatorname{div}(A(h)\nabla z) = 0 \quad \text{in } \Omega,
			\]
			\[
			-\mathbf{n} \cdot A(h)\nabla z + gz = -\beta r' < 0 \quad \text{on } \mathcal T,
			\]
			and
			$z = 0 \quad \text{on } \mathcal B.$
			
			The strong maximum principle implies
			\[
			z > 0 \quad \text{in } \Omega
			\]
			unless \( z \equiv 0 \). The latter alternative is ruled out by the  boundary condition on \( \mathcal T \).
			
			It remains only to rule out a zero on {$\mathcal T$}. Suppose for contradiction that
			$z(q^*, 0) = 0$
			for some \( q^* \in \mathbb{R} \). Since \( z \geq 0 \) and \( z > 0 \) in the interior, the Hopf boundary-point lemma gives
			\[
			-\mathbf{n} \cdot A(h)\nabla z(q^*, 0) > 0.
			\]
			On the other hand, the boundary condition gives
			\[
			-\mathbf{n} \cdot A(h)\nabla z(q^*, 0) = -\beta r'(q^*) < 0,
			\]
			which is a contradiction. Thus
			\[
				h_q = z > 0 \quad {\text{in }\Omega\cup\mathcal T}.
			\]
			{Hopf's lemma on $\mathcal B$ also gives} $h_{qp}>0$ on $\mathcal B$.
			If $\beta=0$, the monotonicity and the endpoint profiles give $H\leq h\leq H$, and hence ${(\beta,w)=(0,0)}$.
		\end{proof}

	\section{{Flow force and compactness}}\label{sec:compactness}
	\subsection{{Flow-force balance and laminar uniqueness}}
	{To control possible loss of compactness at infinity, we use a pressure-adjusted \emph{flow force}, based on the formulations in~\cite{basu,bk,milesjma}. Denote it by $\mathscr S$ and define}
	\begin{equation*}
	\mathscr	S(x)=\int_0^{d+\eta(x)} (P(x,y)-P_{\mathrm{atm}}-R(x)+(c-u(x,y))^2)dy.
	\end{equation*}
	{In height coordinates, with $h=H+w$, this becomes}
	\begin{equation}\label{eq:flow-force}
	\mathscr{S}[h;R](q)=\int_{-m}^0\left(\frac{1-h_q^2}{2h_p^2}+\Gamma-g(h-d)+\frac{\lambda}{2}-R(q)\right)h_pdp.
	\end{equation}

	\begin{lemma}
	For the height function $h$ satisfying \eqref{heighteq}, we have
	\begin{equation}\label{eq:flow-force-balance}
		\frac{d\mathscr S}{dq}[h;R](q)=-R'(q)h(q,0).
	\end{equation}
	In particular, if $R\equiv\mathrm{const}$, then $\mathscr S$ is constant.
\end{lemma}
\begin{proof}
	{Let $I(q,p)$ denote the integrand in \eqref{eq:flow-force}. Equation~\eqref{heighteq1} gives}
	\begin{equation*}
		I_q=\partial_p\left[
		\left(\frac{\lambda}{2}-R(q)+\Gamma(p)-g(h-d)
		-\frac{1+h_q^2}{2h_p^2}\right)h_q
		\right]-R'(q)h_p.
	\end{equation*}
	Together with the boundary conditions on $\mathcal T$ and $\mathcal B$, integrating with respect to $p$ gives
	\begin{equation*}
		\frac{d\mathscr S}{dq}[h;R](q)
		=-R'(q)\int_{-m}^0h_p(q,p)dp
		=-R'(q)h(q,0),
	\end{equation*}
	which completes the proof.
\end{proof}

	{To compare laminar flow forces, let $\hat\lambda$ be a laminar parameter and consider the family}
	\begin{equation*}
		\hat H(p;\hat\lambda)=\int_{-m}^p\frac{ds}{\sqrt{\hat\lambda+2\Gamma(s)}},
	\end{equation*}
	{Its depth is $\hat d=\hat H(0;\hat\lambda)$. Since depth is strictly decreasing in $\hat\lambda$, we may use $\hat d$ as a parameter and write the inverse relation as $\hat\lambda=\tilde\lambda(\hat d)$. Define the corresponding Bernoulli function by}
	\begin{equation*}
		\tilde Q(\hat d)=\frac{\tilde\lambda(\hat d)}{2}+g\hat d.
	\end{equation*}
	{We write $\tilde{\mathscr S}$ for the flow force expressed in this depth parameter:}
	\begin{equation*}
		\tilde{\mathscr S}[\hat H;R](\hat d)
		=\mathscr S[\hat H(\cdot;\tilde\lambda(\hat d));R].
	\end{equation*}

\begin{lemma}\label{lemma:no bore}
	For fixed constant pressure $R^*$, the function
	$\tilde{\mathscr S}[\hat H;R^*](\hat d)$ satisfies
	\begin{equation}\label{bernoulliide}
		\tilde{\mathscr S}[\hat H;R^*]'(\hat d)
		=\tilde Q(d)-R^*-\tilde Q(\hat d).
	\end{equation}
	{Consequently, two distinct laminar flows cannot satisfy the same Bernoulli condition and have the same flow force.}
\end{lemma}
\begin{proof}
	{Substitution into the flow-force formula gives}
	\begin{equation*}
		\begin{aligned}
			\tilde{\mathscr S}[\hat H;R^*](\hat d)
			={}&\int_{-m}^0\sqrt{\tilde\lambda(\hat d)+2\Gamma(p)}dp\\
			&+\frac{\lambda-\tilde\lambda(\hat d)}{2}\hat d
			-\frac{g}{2}\hat d^2+(gd-R^*)\hat d.
		\end{aligned}
	\end{equation*}
	Differentiating with respect to $\hat d$ and applying
	\begin{equation*}
		\frac{d}{d\hat d}
		\int_{-m}^0\sqrt{\tilde\lambda(\hat d)+2\Gamma(p)}dp
		=\frac12\tilde\lambda'(\hat d)
		\int_{-m}^0\frac{dp}{\sqrt{\tilde\lambda(\hat d)+2\Gamma(p)}}
		=\frac12\hat d\,\tilde\lambda'(\hat d),
	\end{equation*}
	we can obtain
	\begin{equation*}
		\tilde{\mathscr S}[\hat H;R^*]'(\hat d)
		=\frac{\lambda}{2}+gd-R^*
		-\frac{\tilde\lambda(\hat d)}{2}-g\hat d
		=\tilde Q(d)-R^*-\tilde Q(\hat d).
	\end{equation*}
	
	Moreover, by \eqref{eq:B-derivative}, we have
	\begin{equation*}
		\tilde Q'(\hat d)
		=g\bigl(1-F^2(\tilde\lambda(\hat d))\bigr),
	\end{equation*}
	which is strictly increasing in $\hat d$. Thus $\tilde Q$ is strictly convex. If two distinct depths $d_1<d_2$ satisfy the same Bernoulli condition
	\begin{equation*}
		\tilde Q(d_1)=\tilde Q(d_2)=\tilde Q(d)-R^*,
	\end{equation*}
	then \eqref{bernoulliide} gives
	\begin{equation*}
		\tilde{\mathscr S}[\hat H;R^*](d_2)
		-\tilde{\mathscr S}[\hat H;R^*](d_1)
		=\int_{d_1}^{d_2}
		\bigl(\tilde Q(d)-R^*-\tilde Q(s)\bigr)ds>0.
	\end{equation*}
	Hence these two laminar flows cannot have the same flow force, which completes the proof.
\end{proof}
\begin{lemma}[Uniqueness at zero pressure]\label{lem:zeropressure}
	  {The only monotone solution on the supercritical component with $\beta=0$ is the laminar flow $w=0$.}
\end{lemma}
	\begin{proof}
		{For $\beta=0$, the pressure and flow force are constant. Lemma~\ref{lemma:no bore} forces the two laminar limits to coincide, and Lemma~\ref{lemma:analytic-map} identifies them with $H$. Monotonicity then gives $h\equiv H$, hence $w=0$.}
	\end{proof}

	\subsection{{Uniform regularity and compactness}} {We first obtain uniform regularity from bounds on the vertical height derivative and the pressure strength. We then combine these estimates with the flow-force balance to control both tails of a sequence of monotone fronts.}
	\begin{lemma}\label{lem:regularity}
		{Let $(\beta,w)\in\mathscr U$ be a zero of $\mathcal F$, and set $h=H+w$. Suppose that, for some constant $M>0$,}
		\begin{equation*}
			\|w_p\|_{L^\infty}+|\beta|\leq M.
		\end{equation*}		
				Then
		\begin{equation*}
			\inf_{\Omega}h_p\ge C^{-1}>0,\qquad
			\|h\|_{C_b^{4+\alpha}(\overline\Omega)}
\le C,
		\end{equation*}
		{where $C$ depends only on $M,\alpha,m,g$, the fixed upstream profile $H$, and the norms $\|\gamma\|_{C^{4+\alpha}([0,m])}$ and $\|r\|_{C_b^{4+\alpha}(\mathbb R)}$.}
	\end{lemma}
\begin{proof}
	In this proof, the constant $C>0$ may change from {line to line. Since} $h_p=H_p+w_p$, the hypothesis gives
\[
    \sup_{\Omega}h_p\le C.
\]
Moreover, $h=0$ on $\mathcal B$ and $h_p>0$, so
\[
    0\le h(q,p)=\int_{-m}^{p}h_p(q,s)\,ds
    \le m\sup_{\Omega}h_p\le C.
\]
Thus $h$ is uniformly bounded in $C^0(\overline\Omega)$.

To obtain gradient bounds throughout the fluid, we use the
pressure minimum-principle argument of \cite{varjde}.
{For the pressure estimate, define the nonnegative vorticity bound}
\[
    a:=\frac12\max\left\{0,\max_{[0,m]}\gamma\right\}.
\]
For the modified pressure
\[
    \tilde P:=P-P_{\mathrm{atm}}+a\psi,
\]
that argument excludes an interior minimum. On the free surface,
$\tilde P=R\ge\inf_{\mathbb R}R$, while on the flat bed the Euler equations give
\[
   \tilde P_y=-g+a\psi_y<0,
\]
which excludes a minimum there as well. At either laminar end,
the limiting pressure is  bounded below by
$P_{\mathrm{atm}}+R^\pm$. Since $\psi\ge0$ and
$R^\pm\ge\inf_{\mathbb R}R$,
{we obtain}
\[
    P-P_{\mathrm{atm}}
    \ge \inf_{\mathbb R}R-a\psi
    \ge -\|R\|_{L^\infty(\mathbb R)}-am
    \qquad\text{in }\Omega_\eta.
\]

Writing the pressure in height coordinates, Bernoulli's law therefore
gives the global estimate
\[
    \frac{1+h_q^2}{2h_p^2}
    =\frac Q2+\Gamma(p)-gh-(P-P_{\mathrm{atm}})
    \le
    \frac Q2+\|\Gamma\|_{L^\infty([-m,0])}
    +\|R\|_{L^\infty(\mathbb R)}+am
    \le C
    \qquad\text{in }\Omega.
\]
Since $h_p>0$, it follows that
\[
    h_p\ge c>0,
    \qquad
    |h_q|\le C h_p\le C
    \qquad\text{in }\Omega.
\]
Combining these estimates, we obtain
\[
    \|h\|_{C_b^1(\overline\Omega)}
    +\frac{1}{\inf_{\Omega}h_p}
    \le C.
\]

	{We first derive a uniform $C_b^{2+\alpha_0}$ estimate, where $\alpha_0\in(0,\alpha]$ is an exponent supplied by the boundary regularity estimate below.} The height equation in non-divergence form \cite{constantincpam04} is
	\[
h_p^2 h_{qq} - 2h_p h_q h_{pq} + (1 + h_q^2) h_{pp} + \gamma(-p) h_p^3 = 0 \quad \text{in } \Omega,
\]
while the boundary condition on the top boundary $\mathcal T$ is
\[
\frac{1 + h_q^2}{2h_p^2} + g(h - d) - \frac{\lambda}{2} + \beta r(q) = 0 \quad \text{on } \mathcal T.
\]
Fix \( q_0 \in \mathbb{R} \) and a radius \( \rho \in (0, m/4) \), independent of \( q_0 \), and {consider the half-ball adjacent to the top boundary}
\[
B_\rho^- (q_0) := B_\rho ((q_0, 0)) \cap \Omega.
\]
The bounds $c\leq h_p\leq C$ and $|h_q|\leq C$ give uniform ellipticity: the principal matrix
\[
\begin{pmatrix}
h_p^2&-h_ph_q\\
-h_ph_q&1+h_q^2
\end{pmatrix}
\]
has determinant $h_p^2\geq c^2$ and bounded trace. The derivative of the displayed Bernoulli boundary expression with respect to its normal-gradient argument $h_p$ is
\[
-\frac{1+h_q^2}{h_p^3},
\]
whose absolute value is bounded below by a positive constant. Thus the top boundary condition is uniformly oblique. On a fixed neighborhood of the controlled range of $(h,Dh)$, the differential expression is affine in the Hessian, and its derivatives with respect to the gradient grow at most linearly in the Hessian. Its remaining first derivatives are bounded. Moreover,
\[\|\beta r\|_{C_b^{1+\alpha}(\R)}\leq C,\]
and the boundary expression has uniformly bounded $C^{1+\alpha}$ norms as a function of its arguments. These properties verify the ellipticity, boundary nondegeneracy, and structural conditions (3.1)--(3.2) in \cite[Theorem 3]{lieberman87}; the operators may be smoothly modified away from the controlled range while preserving these bounds. The boundary is flat, and $\rho$ is fixed. The localized estimate in that theorem therefore gives, for some $\alpha_0\in(0,\alpha]$ independent of $q_0$ and of the solution,
\begin{equation*}
\|h\|_{C^{2+\alpha_0}(\overline{B_{\rho/2}^-(q_0)})} \leq C.
\end{equation*}

At the bottom $\mathcal B$, we retain the homogeneous Dirichlet condition $h=0$. The localized boundary gradient H\"older estimate in \cite[Theorem 13.7 and (13.41)]{Ts} gives a uniform $C^{1+\alpha_0}$ bound on smaller bottom half-balls, after decreasing $\alpha_0$ if necessary. The interior gradient estimate in \cite[Theorem 13.6]{Ts} gives the same bound on smaller interior balls. These estimates use the established $C^1$ bound and the controlled coefficient derivatives; the boundary data and the geometry are fixed.

On each such smaller neighborhood, regard the height equation as a linear non-divergence equation for $h$, with the principal coefficients evaluated at $Dh$ and right-hand side $-\gamma(-p)h_p^3$. The $C^{1+\alpha_0}$ bound just obtained makes these coefficients and the right-hand side uniformly bounded in $C^{\alpha_0}$. The interior and Dirichlet boundary Schauder estimates \cite[Chapter 6]{Ts}, on a further smaller neighborhood, now give the required $C^{2+\alpha_0}$ bound. In particular, Theorem 13.6 supplies the gradient estimate, and the second-derivative estimate follows from this additional Schauder step. Combining these estimates with the top estimate and translating a fixed covering along the strip yields
\[
\|h\|_{C_b^{2+\alpha_0}(\bar{\Omega})} \leq C.
\]
The local radii and overlap bounds are fixed, so the global H\"older seminorm is controlled as well. All constants depend only on the quantities listed in the lemma; no bound on the initial $C_b^{3+\alpha}$ norm of $h$ has been used.
Set
\[z=w_q=h_q,\]
{and differentiate the height equation with respect to $q$. Then $z$ satisfies}
\begin{align*}
&\operatorname{div}(A(h)\nabla z) = 0 \quad &&\text{in } \Omega,\\
&-\mathbf n\cdot(A(h)\nabla z)+gz=-\beta r'(q)\quad&&\mathrm{on~}\mathcal T,\\
&z=0&&\mathrm{on~}\mathcal B,
\end{align*}
where $A(h)$ is given in \eqref{eq:principal-matrix}.
The preceding estimates imply that the coefficients of this mixed boundary value problem have uniformly bounded \( C_b^{1+\alpha_0} \) norms. The operator is uniformly elliptic and the top boundary operator is uniformly oblique by the estimates established earlier. The Schauder estimate in \cite{agmon} gives
\begin{equation*}
	\|z\|_{C_b^{2+\alpha_0}(\bar{\Omega})} \leq C \left( \|z\|_{C^0(\bar{\Omega})} + |\beta| \|r'\|_{C_b^{1+\alpha_0}(\mathbb{R})} \right).
\end{equation*}
{Solving the height equation for $h_{pp}$ gives, in $\Omega$,}
\begin{equation}\label{eq:hpp}
h_{pp} = \frac{2h_p h_q h_{pq} - h_p^2 h_{qq} - \gamma(-p)h_p^3}{1 + h_q^2}.
\end{equation}
{Since $h_q\in C_b^{2+\alpha_0}$ and $h\in C_b^{2+\alpha_0}$, the right-hand side of \eqref{eq:hpp} is bounded in $C_b^{1+\alpha_0}$. Thus $h\in C_b^{3+\alpha_0}$ and consequently $A(h)\in C_b^{2+\alpha_0}\subset C_b^{1+\alpha}$. A second application of the mixed-boundary Schauder estimate to the equation for $z=h_q$ gives
\begin{equation*}
	\|h_q\|_{C_b^{2+\alpha}(\overline\Omega)}\leq C\bigl(\|h_q\|_{C^0(\overline\Omega)}+|\beta|\|r'\|_{C_b^{1+\alpha}(\R)}\bigr).
\end{equation*}
Substitution into \eqref{eq:hpp} yields $h\in C_b^{3+\alpha}$ and hence $A(h)\in C_b^{2+\alpha}$. Applying the Schauder estimate once more gives
\begin{equation*}
	\|h_q\|_{C_b^{3+\alpha}(\overline\Omega)}\leq C\bigl(\|h_q\|_{C^0(\overline\Omega)}+|\beta|\|r'\|_{C_b^{2+\alpha}(\R)}\bigr).
\end{equation*}
{A further use of \eqref{eq:hpp} gives $h_{pp}\in C_b^{2+\alpha}$ and hence}
\[\|w\|_{C_b^{4+\alpha}(\overline\Omega)}\leq C.\]}
{Since $h=H+w$ and $H$ is fixed, the same bound holds for $h$. This completes the proof.}
\end{proof}	
	
\begin{lemma}[Compactness of bounded monotone fronts]\label{lem:compactness}
{In the supercritical setting $\sigma_0^\pm<0$, let $(\beta_n,w_n)\in\mathcal F^{-1}(0)$ be a sequence of monotone solutions.}  Assume that for some \(M,\delta>0\),
\begin{equation*}
 \|w_n\|_{C_b^{3+\alpha}}+|\beta_n|\le M,\qquad
 \inf_{\overline\Omega}(H_p+w_{np})\ge\delta,\qquad
 F_+(\beta_n)\ge 1+\delta.
\end{equation*}
{Then a subsequence converges in $\R\times\mathscr X_\infty$ to a zero of $\mathcal F$.}	
\end{lemma}
\begin{proof}
	{Let $h_n$ be the height function and $K_n^+$ its downstream laminar profile:}
	\begin{equation*}
		h_n(q,p)=H(p)+w_n(q,p),\qquad K_n^{+}(p):=H(\cdot;\lambda_+(\beta_n)).
	\end{equation*}
	{After passing to a subsequence, assume that $\beta_n\to\beta$. Since the mapping $\beta\mapsto\lambda_+(\beta)$ is analytic and the downstream states stay uniformly supercritical, their profiles converge to the downstream profile $K^+:=H(\cdot;\lambda_+(\beta))$ of the limiting parameter:}
	\begin{equation*}
		K_n^{+}\to K^+:=H(\cdot;\lambda_+(\beta))\in C^{4+\alpha}([-m,0]).
	\end{equation*}	
	Lemma \ref{lem:regularity} gives
	\begin{equation}\label{eq:c4est}
		\sup_n\|h_n\|_{C_b^{4+\alpha}(\overline{\Omega})}\leq C(M,\delta).
	\end{equation}
	We can extract a subsequence  such that
	\begin{equation}
		\label{coninbdd}
		h_n\to h\quad\mathrm{in~}C_{\mathrm{loc}}^{3+\alpha}(\overline{\Omega}),
	\end{equation}
	{The limit $h$ solves \eqref{heighteq} with parameter $\beta$. If $\beta=0$, monotonicity and the endpoint conditions give}
	\begin{equation*}
		\min\{H,K_n^+\}\leq h_n(q,\cdot)\leq\max\{H,K_n^+\}
	\end{equation*}
in the pointwise sense. Lemma \ref{lem:downstream-state} gives $K_n^+\to H$ and thus $h_n\to H$ {globally in the $C^0$ norm}. Together with \eqref{eq:c4est}, we obtain the convergence in $C_b^{3+\alpha}$ and $\mathscr{X}_\infty$.

	 {It remains to consider $\beta\ne0$. By Section~\ref{mono},}
	\begin{equation*}
		\sgn(\beta_n)(h_n-H)\geq 0,\quad\sgn(\beta_n)(K_n^+-h_n)\geq0,\quad\sgn(\beta_n)h_{nq}\geq0
	\end{equation*}
	throughout $\Omega\cup\mathcal T$.
	
	{We first prove that convergence to the downstream profile is uniform in $n$. Let $\ell>0$ be a horizontal cutoff distance. We claim that}
	\begin{equation}\label{uniformdown}
		{\lim_{\ell\to\infty}\sup_n\sup_{q\geq\ell,p\in[-m,0]}}|h_n(q,p)-K_n^+(p)|=0.
	\end{equation}
	{If this failed, there would be $\epsilon_0>0$ and sequences $p_n\to p_*\in[-m,0]$ and $\bar q_n\to\infty$ such that}
	\begin{equation*}
		\sgn(\beta_n)(K_n^+(p_n)-h_n(\bar q_n,p_n))\geq2\epsilon_0.
	\end{equation*}
	By the monotonicity of $q\to h_n(q,p_n)$ and
	\begin{equation*}
		\lim_{q\to\infty}h_n(q,p_n)=K_n^+(p_n),
	\end{equation*}
	{for each $n$, there is $q_n>\bar q_n$ such that}
	\begin{equation}\label{eq:cross}
		\sgn(\beta_n)(K_n^+(p_n)-h_n( q_n,p_n))=\epsilon_0.
	\end{equation}
	{Since both profiles vanish at the bed, the uniform $C^1$ bound gives}
	\begin{equation*}
		\epsilon_0=\sgn(\beta_n)(K_n^+(p_n)-h_n( q_n,p_n))\leq C(p_n+m).
	\end{equation*}
	{Thus $p_*>-m$. Define the translated heights by}
	\begin{equation*}
		\tilde{h}_n(q,p):=h_n(q+q_n,p).
	\end{equation*}
	{By \eqref{eq:c4est}, a subsequence converges to a limiting height $k$:}
	\begin{equation*}
		\tilde h_n\to k(q,p)\quad\mathrm{in~}C_{\mathrm{loc}}^{3+\alpha}(\overline\Omega).
	\end{equation*}
	Since $q_n\to\infty$, thanks to Assumption \ref{ass:pressure},
	\begin{equation*}
		\beta_n r(q+q_n)\to\beta\quad\mathrm{in~}C_{\mathrm{loc}}^{4}(\R).
	\end{equation*}
	{Consequently, $k$ solves \eqref{heighteq} with constant pressure $\beta$ and satisfies $\sgn(\beta)k_q\geq0$. Boundedness and monotonicity give a left endpoint profile, denoted by $K^l$, such that}
	\begin{equation*}
		\lim_{q\to-\infty}k(q,p)=K^l(p).
	\end{equation*}
	{To identify this endpoint and its regularity, take any sequence $s_j\to-\infty$. The translates $k(q+s_j,p)$ have a subsequence converging in $C_{\mathrm{loc}}^{3+\alpha}$. Monotonicity forces every such limit to be independent of $q$ and equal to $K^l$. Therefore,}
	\begin{equation*}
		k(q,p)\to K^l(p)\quad \mathrm{in~} C^3([-m,0]),\mathrm{~as~}q\to-\infty,
	\end{equation*}
	{and $K^l$ is a laminar solution with the same total head $Q$ and constant pressure $\beta$. Equation~\eqref{eq:cross} and monotonicity imply}
	\begin{equation}\label{r+neq}
		\sgn(\beta)(K^+(p_*)-K^l(p_*))\geq \epsilon_0.	\end{equation}
	{We now compare the flow forces to obtain a contradiction. For each fixed $x\in\R$, integrating \eqref{eq:flow-force-balance} and using Assumption~\ref{ass:pressure} gives}
	\begin{equation*}
		\mathscr S[h_n;\beta_n r](q_n+x)-\mathscr{S}[K_n^+;\beta_n]=\int_{q_n+x}^\infty\beta_n r'(s)h_n(s,0)ds\to0\quad\mathrm{as~}n\to\infty.
	\end{equation*}
That is,
\begin{equation*}
	\mathscr{S}[k;\beta](x)=\mathscr{S}[K^+;\beta]
\end{equation*}
	for any $x\in\R$. {Letting $x\to-\infty$ gives}
	\[\mathscr{S}[K^l;\beta]=\mathscr{S}[K^+;\beta].\]
	By Lemma \ref{lemma:no bore}, we have $K^+=K^l$, which contradicts \eqref{r+neq},
	so \eqref{uniformdown} holds.
	
	{The upstream argument is similar. If convergence to $H$ were not uniform, there would be $\epsilon_1>0$ and sequences $\hat q_n\to-\infty$ and $p_n\to p_*$ such that}
	\begin{equation*}
		\sgn(\beta_n)(h_n(\hat q_n,p_n)-H(p_n))\ge\epsilon_1.
	\end{equation*}
	{For each $n$, monotonicity gives a point $q_n\leq\hat q_n$ such that}
	\begin{equation}\label{r-neq}
		\sgn(\beta_n)(h_n(q_n,p_n)-H(p_n))=\epsilon_1,
	\end{equation}
	and we also have $p_*>-m$. Let
	\begin{equation*}
		\hat h_{n}(q,p):=h_n(q+q_n,p).
	\end{equation*}
	{The same compactness argument gives a limiting height $k$, after passing to a subsequence. Meanwhile,}
	\begin{equation*}
		\beta_nr(q+q_n)\to 0\quad\mathrm{in~}C_{\mathrm{loc}}^{4}(\R).
	\end{equation*}
    {Thus $k$ solves \eqref{heighteq} with zero pressure. Denote its right endpoint profile by $K^r$; the preceding translation argument gives}
    \begin{equation*}
	    	\lim_{q\to\infty}k(q,p)=K^r(p) \quad\mathrm{in~}C^3([-m,0]).
    \end{equation*}
	{The profile $K^r$ is a laminar solution with total head $Q$ and zero pressure. Equation~\eqref{r-neq} and monotonicity give}
	\begin{equation*}
		\sgn(\beta)(K^r(p_*)-H(p_*))\geq \epsilon_1.
	\end{equation*}
	{Integrating \eqref{eq:flow-force-balance} from the upstream end and using Assumption~\ref{ass:pressure} gives}
	\begin{equation*}
\mathscr{S}[k;0](x)=\mathscr{S}[H;0]
	\end{equation*}
	for any $x\in\R$.
 {Letting $x\to\infty$ gives}
	\[\mathscr{S}[K^r;0]=\mathscr{S}[H;0].\]
	By Lemma \ref{lemma:no bore}, we have $K^r=H$, which contradicts \eqref{r-neq}.
	{Therefore,}
	\begin{equation*}
		{\lim_{\ell\to\infty}\sup_n\sup_{q\leq-\ell,p\in[-m,0]}}|h_n(q,p)-H(p)|=0.
	\end{equation*}
	
	{Uniform convergence at both ends and the local convergence \eqref{coninbdd} imply $h_n\to h$ globally in the $C^0$ norm. To improve this convergence, we use the H\"older interpolation inequality for regularity exponents $0<s<t$:}
	\begin{equation}
		\label{interpholdero}
\|u\|_{C_b^s(\overline\Omega)}\leq C\|u\|_{C^0(\overline\Omega)}^{1-s/t}\|u\|^{s/t}_{C_b^t(\overline\Omega)}.
	\end{equation}
{Taking $u=h_n-h$, $s=3+\alpha$, and $t=4+\alpha$ in \eqref{interpholdero}, and writing $\theta=1/(4+\alpha)$, gives}
\begin{equation*}
	 \|h_n-h\|_{C_b^{3+\alpha}}
 \le C\|h_n-h\|_{C^0}^{\,\theta}
       \bigl(\|h_n\|_{C_b^{4+\alpha}}
             +\|h\|_{C_b^{4+\alpha}}\bigr)^{1-\theta},
 \qquad \theta=\frac1{4+\alpha},
\end{equation*}
{The global $C^0$ convergence and the uniform $C_b^{4+\alpha}$ bound therefore yield convergence in $C_b^{3+\alpha}$. The endpoint profiles converge with the same regularity, so $(\beta_n,w_n)$ converges in $\R\times\mathscr X_\infty$. Continuity of $\mathcal F$ shows that the limit is a zero of $\mathcal F$.}
\end{proof}

	\section{{Global continuation and large-amplitude fronts}}\label{sec:global-continuation}
	We use the following analytic global implicit-function theorem.  In the form
needed here, it is \cite[Theorem~B.1]{ChenWalshWheelerFronts}, which is a version
of the Dancer--Buffoni--Toland continuation theorem \cite{BuffoniToland} that retains loss of
compactness and loss of Fredholmness as explicit alternatives.

\begin{theorem}\label{thm:maximal curve}
Let \(\mathcal{W}\) and \(\mathcal{Z}\) be Banach spaces, and \(\mathcal{U} \subset \mathcal{W} \times \mathbb{R}\) an open set containing a point \((w_0, \lambda_0)\). Suppose that \(\mathcal{G}: \mathcal{U} \to \mathcal{Z}\) is real-analytic and satisfies
\[
\mathcal{G}(w_0, \lambda_0) = 0, \qquad
\mathcal{G}_w(w_0, \lambda_0): \mathcal{W} \to \mathcal{Z} \text{ is an isomorphism.} \tag{B.1}
\]
Then there exists a curve \(\mathcal{K}\) that admits the global \(C^0\) parameterization
\[
\mathcal{K} = \{ (w(s), \lambda(s)) : s \in \mathbb{R} \} \subset \mathcal{G}^{-1}(0) \cap \mathcal{U},
\]
and satisfies the following:
\begin{enumerate}
	\item[(a)] {For each $s\in\mathbb R$, the linearized operator} \(\mathcal{G}_w(w(s), \lambda(s)): \mathcal{W} \to \mathcal{Z}\) is Fredholm of index \(0\).
	\item[(b)] {At least one of the following alternatives holds in each direction, as $s\to+\infty$ and as $s\to-\infty$:}
	\begin{itemize}
		{\item[(A1)] (Blowup) The following quantity measures the size of the solution and its inverse distance to the boundary of $\mathcal U$:}
		\[
		N(s) := \| w(s) \|_{\mathcal{W}} + |\lambda(s)| + \frac{1}{\operatorname{dist}((w(s), \lambda(s)), \partial \mathcal{U})} \to \infty.
		\]
		\item[(A2)] (Loss of compactness) There exists a sequence \(s_n \to \pm\infty\) such that \(\sup_n N(s_n) < \infty\), but \((w(s_n), \lambda(s_n))\) has no convergent subsequence in \(\mathcal{W} \times \mathbb{R}\).
		\item[(A3)] (Loss of Fredholmness) There exists a sequence \(s_n \to \pm\infty\) such that \(\sup_n N(s_n) < \infty\) and \((w(s_n), \lambda(s_n)) \to (w_*, \lambda_*) \in \mathcal{W} \times \mathbb{R}\), but \(\mathcal{G}_w(w_*, \lambda_*)\) is not Fredholm of index \(0\).
		\item[(A4)] (Closed loop) There exists \(T > 0\) such that \((w(s+T), \lambda(s+T)) = (w(s), \lambda(s))\) for all \(s \in (0, \infty)\).
	\end{itemize}
	\item[(c)] Near each point \((w(s_0), \lambda(s_0)) \in \mathcal{K}\), we can locally reparameterize \(\mathcal{K}\) so that \(s \mapsto (w(s), \lambda(s))\) is real-analytic.
	\item[(d)] The curve \(\mathcal{K}\) is maximal in the sense that if \(\mathcal{J} \subset \mathcal{G}^{-1}(0) \cap \mathcal{U}\) is a locally real-analytic curve containing \((w_0, \lambda_0)\) and along which \(\mathcal{G}_w\) is Fredholm of index \(0\), then \(\mathcal{J} \subset \mathcal{K}\).
	\end{enumerate}
\end{theorem}
{Let $\mathcal C_{\mathrm{loc}}$ denote the local solution curve constructed in Theorem~\ref{thm:local}:}
\begin{equation}\label{eq:localcurve}
	\mathcal{C}_\mathrm{loc}=\{(\beta,w(\beta)):|\beta|<\varepsilon\},
\end{equation}
	{For $\varepsilon>0$ sufficiently small, Theorem~\ref{thm:local} gives $\beta h_q>0$ whenever $0<|\beta|<\varepsilon$. Denote the two nontrivial half-branches of \eqref{eq:localcurve} by}
	\begin{equation*}
		\mathcal{C}_{\mathrm{loc}}^+=\mathcal{C}_{\mathrm{loc}}\cap\{\beta>0\}\mathrm{~and~}\mathcal{C}_{\mathrm{loc}}^-=\mathcal{C}_{\mathrm{loc}}\cap\{\beta<0\}.
	\end{equation*}
	
{Theorem~\ref{thm:maximal curve} extends the local curve to a maximal continuous curve, denoted by}
\begin{equation} \label{zyy}
	\mathcal{C}=\{(\beta(s),w(s)):s\in\R\}
\end{equation}
	satisfying
	\begin{equation*}
		(\beta(0),w(0))=(0,0)\mathrm{~and~}\mathcal{F}(\beta(s),w(s))=0.
	\end{equation*}
Both endpoint states remain supercritical along $\mathcal C$: their laminar parameters vary continuously and cannot cross $\lambda_{\rm cr}$, since the upstream Bernoulli value is fixed at $\mathcal Q(\lambda)>\mathcal Q(\lambda_{\rm cr})$, while a critical downstream state would force $\beta=\beta_{\rm cr}$, excluded by $\mathscr U$.
	{Choose the parameter near the origin so that $\beta(s)=s$, and orient the global curve accordingly. For $s\geq0$, write $(\beta^\pm(s),w^\pm(s))=(\beta(\pm s),w(\pm s))$ for the two directions. Choose $s_0>0$ sufficiently small that}
	\begin{equation*}
		{\beta^+(s)>0,\qquad\beta^-(s)<0,\qquad 0<s<s_0.}
	\end{equation*}
	{After shortening the local curve if necessary, its half-branches can also be written as}
\begin{equation*}
	{\begin{aligned}
\mathcal C_{\mathrm{loc}}^+&=\{(\beta^+(s),w^+(s)):0<s<s_0,\ \beta^+(s)>0\},\\
\mathcal C_{\mathrm{loc}}^-&=\{(\beta^-(s),w^-(s)):0<s<s_0,\ \beta^-(s)<0\}.
\end{aligned}}
\end{equation*}
	and we denote their maximal continuations, together with $(0,0)$, by $\mathcal C^+$ and $\mathcal C^-$, respectively. {Proposition~\ref{prop:mono} below proves that their signs persist globally.} By \eqref{nodaldep} and \eqref{nodalele}, we have the following proposition.

\begin{prop}	[Global monotonicity and exclusion of closed loops]\label{prop:mono}
	{Let $h^\pm(s)=H+w^\pm(s)$ be the height profiles along $\mathcal C^\pm$. Then, for every $s>0$,}
	\begin{equation}\label{eq:nodalele}
		\beta^+(s)>0,\quad h_q^+(s)>0\mathrm{~on~}\Omega\cup\mathcal{T},
			\end{equation}
			and
			\begin{equation}\label{eq:nodaldep-global}
				\beta^-(s)<0,\quad h_q^-(s)<0\mathrm{~on~}\Omega\cup\mathcal{T}.
			\end{equation}
			{Writing $K_+(\beta):=H(\cdot;\lambda_+(\beta))$ for the downstream profile, we have, on $\mathcal C^+$,}
			\begin{equation*}
				H(p)< h^+(q,p)< {K_+(\beta^+)}(p),
			\end{equation*}
			and on $\mathcal{C}^-$, we have
			\begin{equation*}
				{K_+(\beta^-)}(p)<h^-(q,p)<H(p)
			\end{equation*}
			{for every $q\in\R$ and $p\in(-m,0]$. In particular, $\mathcal C^+\cap\mathcal C^-=\{(0,0)\}$, and neither branch forms a closed loop.}
\end{prop}	
	\begin{proof}
		{Let $\mathcal N^+$ be the subset of $\mathcal C^+$ satisfying the positive nodal properties:}
		\begin{equation*}
			\mathcal{N}^+ = \left\{ (\beta,w) \in \mathcal{C}^+ : \beta > 0, \quad h_q > 0 \text{ in } \Omega \cup T, \quad h_{pq} > 0 {\text{ on }\mathcal B} \right\}.
		\end{equation*}
		{Theorem~\ref{thm:local} and Hopf's lemma give $\mathcal N^+\ne\emptyset$.} Moreover, Lemma \ref{lem:open} and Lemma \ref{lem:closed} imply that $\mathcal N^+$ is both relatively open and relatively closed away from $(0,0)$. Hence it contains the continuation of $\mathcal C_{\rm loc}^+$ up to a possible return to $(0,0)$. Likewise, we define
		\[\mathcal{N}^- = \left\{ (\beta,w) \in \mathcal{C}^- : \beta < 0, \quad h_q < 0 \text{ in } \Omega \cup T, \quad h_{pq} < 0 {\text{ on }\mathcal B} \right\},\]
		and obtain the same conclusion for $\mathcal C^-$. Lemma \ref{lem:zeropressure} shows that the only possible return point with $\beta=0$ is $(0,0)$. We now rule out such a return. {Local uniqueness at $(0,0)$ gives a neighborhood in which the nontrivial zero set consists of exactly two half-branches. Denoting them by $\Gamma_{\mathrm{loc}}^\pm$, we have, within this local zero set,}
 \begin{equation*}
 	\Gamma^+_{\mathrm{loc}}=\{\beta>0,h_q>0\}\mathrm{~and~}\Gamma^-_{\mathrm{loc}}=\{\beta<0,h_q<0\}.
 	\end{equation*}
	{Suppose the global analytic curve starting from $\Gamma^+_{\mathrm{loc}}$ forms a loop and returns to the origin. Since $\mathcal F_w(0,0)$ is an isomorphism, the implicit function theorem identifies the local zero set with a single embedded analytic arc having exactly the two half-branches $\Gamma^+_{\mathrm{loc}}$ and $\Gamma^-_{\mathrm{loc}}$. We use the no-retracing parameterization furnished by the distinguished-arc construction underlying Theorem \ref{thm:maximal curve}: before a loop closes, the route does not traverse a nonempty open subarc already traversed. Indeed, at a first such overlap the two local analytic normalizations agree on an open subarc, and uniqueness of analytic continuation identifies the incoming distinguished arc with the earlier one, contradicting the first-overlap choice. Hence, if the route leaves $(0,0)$ through $\Gamma^+_{\mathrm{loc}}$ and later closes there, its terminal local segment cannot lie in $\Gamma^+_{\mathrm{loc}}$, since that would retrace the initial local segment. It must therefore enter through $\Gamma^-_{\mathrm{loc}}$. However, the construction of $\mathcal N^+$ shows that every nontrivial solution on the positive curve satisfies}
 	\[\beta>0,\quad h_q>0.\]
	{Thus it can meet only $\Gamma_{\mathrm{loc}}^+$ and cannot return through $\Gamma_{\mathrm{loc}}^-$. This contradiction rules out a return of $\mathcal C^+$ to the origin; the same argument applies to $\mathcal C^-$.} Thus \eqref{eq:nodalele} and \eqref{eq:nodaldep-global}  hold throughout the two curves, and the proof is complete.
	\end{proof}

{We now relate the abstract alternatives in Theorem \ref{thm:maximal curve} to the explicit exhaustion below. If the abstract exhaustion is bounded, then $|\beta|+\|w\|_{C_b^{3+\alpha}}$ is bounded and the distance from $\partial\mathscr U$ is positive; hence $\inf_\Omega h_p$ and $\beta_{\rm cr}-\beta$ are bounded below. By Lemma \ref{lem:downstream-state}, this also gives $F_+(\beta)-1\geq\delta>0$. Lemma \ref{lem:compactness} then rules out (A2). Alternative (A3) is also impossible: its limit remains in $\mathscr U$, and Proposition \ref{prop:fredholm} says that the limiting linearization is Fredholm of index zero. Alternative (A4) was excluded above. Conversely, boundedness of the explicit quantities in \eqref{eq:blowupfunc} keeps both defining inequalities of \eqref{eq:U-open} uniformly strict (using again the strict monotonicity in Lemma \ref{lem:downstream-state}), and therefore keeps a positive distance from $\partial\mathscr U$. Thus the only remaining alternative, (A1), is equivalently {expressed by the divergence of the following explicit exhaustion function $N(s)$:}}
	\begin{equation}\label{eq:blowupfunc}
	N(s) = \|w(s)\|_{C_b^{3+\alpha}} + |\beta(s)| + \frac{1}{\inf_{\Omega} h_p(s)} + \frac{1}{F_+(\beta(s)) - 1}\to\infty.
	\end{equation}
{The conjugate flow equation and the regularity estimate now reduce \eqref{eq:blowupfunc} to the alternatives in our main theorem.}

\begin{theorem}[Main result]\label{thm:main}
Under the assumptions of Theorem~\ref{thm:local}, denote by $\mathcal C^+$ and $\mathcal C^-$ the positive and negative branches, respectively, of the global solution curve $\mathcal C$ given in \eqref{zyy}.
Then we have
\begin{equation}\label{eq:exhaustblowup}
	\sup_{(\beta,w) \in \mathcal C^\pm} N(\beta,w) = \infty,
\end{equation}
where $N$ is the exhaustion function defined in \eqref{eq:blowupfunc}. Moreover, every nontrivial solution on $\mathcal C^-$ has the depression properties, and at least one of the following alternatives holds:
\begin{enumerate}
	\item [$(i^-)$](Loss of uniform unidirectionality) $\sup_{{(\beta,w)\in\mathcal C^-}} \|h_p\|_{C^0(\overline\Omega)} = \infty$;
	\item [$(ii^-)$] $\inf_{(\beta,w)\in\mathcal C^-}\beta=-\infty$.
\end{enumerate}

{Every nontrivial solution on $\mathcal C^+$ has the elevation properties, and at least one of the following alternatives holds:}
\begin{enumerate}
		\item [$(i^+)$](Loss of uniform unidirectionality) $\sup_{{(\beta,w)\in\mathcal C^+}} \|h_p\|_{{C^0(\overline\Omega)}} = \infty$;
		\item [$(ii^+)$] $\inf_{(\beta,w) \in \mathcal{C}^+} (F_+( \beta ) - 1) = 0$.
\end{enumerate}
\end{theorem}
	\begin{proof}
It follows from \eqref{eq:blowupfunc} that \eqref{eq:exhaustblowup} holds obviously.
		For the negative branch $\mathcal{C}^-$,  the depression property is given by Proposition \ref{prop:mono}. Assume that neither  $(i^-)$ nor $(ii^-)$ holds. {Then there is a constant $M>0$ such that, for every nontrivial solution on $\mathcal C^-$,}
		\begin{equation*}
			\|w_p\|_{C^0(\overline\Omega)}\leq M \mathrm{~and~}0>\beta\geq -M.
		\end{equation*}
		Thanks to Lemma \ref{lem:downstream-state}, we have
		\begin{equation*}
			F_+(\beta)>F_+(0)>1.
		\end{equation*}
		Therefore, there exists some $\delta>0$ such that
		\begin{equation*}
			F_+(\beta)\geq 1+\delta.
		\end{equation*}
		Together with Lemma \ref{lem:regularity}, we  obtain
		\[\inf_{\Omega} h_p \geq c> 0, \quad \|w\|_{C_b^{3+\alpha}} \leq C.\]
		Therefore, we have $\sup_{\mathcal C^-} N(\beta,w) < \infty$, which contradicts \eqref{eq:exhaustblowup}. {For $\mathcal C^+$, failure of $(ii^+)$ bounds $F_+-1$ away from zero, while $0<\beta<\beta_{\rm cr}$ bounds the pressure strength. If $(i^+)$ also fails, the same regularity estimate again bounds the exhaustion function, giving the required contradiction.}
	\end{proof}
\begin{remark}
Since $c-u=1/h_p$, the alternatives $(i^\pm)$ mean that the horizontal relative velocity approaches zero along a sequence of solutions, that is, these branches lose uniform unidirectionality. In alternative $(ii^-)$, $d_+(\beta)\to0$ along any sequence with $\beta\to-\infty$. In alternative $(ii^+)$, the downstream flow approaches a critical state along a sequence on the branch.
\end{remark}
		
\section{An explicit example satisfying Assumption~\ref{ass:pressure}}
	
Here we give an explicit example, where the resulting pressure profile satisfies Assumption~\ref{ass:pressure}. In this construction the surface and pressure are determined from a chosen stream function.
	
	\begin{example}\label{lem:explicit-front}
		{Let $\gamma(\psi)=\gamma_0$ be constant, fix the upstream depth $d>0$, and let $\mu>0$ denote the magnitude of the upstream relative horizontal velocity at the bed. Write $a$ for the corresponding magnitude at the free surface, and assume}
		\begin{equation}\label{eq:explicit-supercritical}
			a:=\mu+\gamma_0d>0,
			\qquad
			\mu(\mu+\gamma_0d)>gd.
		\end{equation}
		{The corresponding relative mass flux is}
		\begin{equation*}
			m=\mu d+\frac{\gamma_0d^2}{2}.
		\end{equation*}
		{For $k>0$, let $\Theta_k$ denote the imaginary part of the principal logarithm of $1+e^{k(x+iy)}$, as specified below. There is a sufficiently small $k>0$ such that, for every sufficiently small nonzero amplitude $\varepsilon$, the stream function}
		\begin{equation}\label{eq:explicit-stream}
			\psi_\varepsilon(x,y)
			=m-\mu y-\frac{\gamma_0y^2}{2}
			-\frac{\varepsilon}{k}\Theta_k(x,y),
		\end{equation}
		where
		\begin{equation}\label{eq:explicit-theta}
			\Theta_k(x,y)
			:=\Im\log\bigl(1+e^{k(x+iy)}\bigr)
			=\arctan\left(\frac{e^{kx}\sin(ky)}
			{1+e^{kx}\cos(ky)}\right),
		\end{equation}
		determines a non-stagnant front.  More precisely, there is a unique
		{free surface $y=\eta_\varepsilon(x)$ on which $\psi_\varepsilon=0$. In this example, $\eta_\varepsilon$ denotes the full surface height above the bed; the surface elevation in the earlier notation is $\eta_\varepsilon-d$. Denoting the downstream depth by $d_+(\varepsilon)$, its endpoint depths satisfy}
		\begin{equation}\label{eq:explicit-depth-limits}
			\lim_{x\to-\infty}\eta_\varepsilon(x)=d,
			\qquad
			\lim_{x\to+\infty}\eta_\varepsilon(x)=d_+(\varepsilon),
		\end{equation}
		where $d_+(\varepsilon)$ is determined by
		\begin{equation}\label{eq:explicit-downstream-depth}
			m=(\mu+\varepsilon)d_+(\varepsilon)
			+\frac{\gamma_0d_+(\varepsilon)^2}{2}.
		\end{equation}
Here uniqueness refers to the surface near $y=d$ satisfying $-(\psi_\varepsilon)_y>0$ in the fluid, and $d_+(\varepsilon)$ denotes the root tending to $d$ as $\varepsilon\to0$.
		Moreover,
		\begin{equation}\label{eq:explicit-monotone-surface}
			\sgn \eta_\varepsilon'(x)=-\sgn\varepsilon.
		\end{equation}
		
		{Define the surface pressure perturbation $R_\varepsilon$ from Bernoulli's law by}
		\begin{equation}\label{eq:explicit-pressure}
			R_\varepsilon(x)
			:=\frac{a^2}{2}+gd
			-\frac12\left|\nabla\psi_\varepsilon
			\bigl(x,\eta_\varepsilon(x)\bigr)\right|^2
			-g\eta_\varepsilon(x)
		\end{equation}
		{and define its downstream value $\beta_\varepsilon$ by}
		\begin{equation*}
			\beta_\varepsilon
			:=\frac{a^2}{2}+gd
			-\frac12\bigl(\mu+\varepsilon
			+\gamma_0d_+(\varepsilon)\bigr)^2
			-gd_+(\varepsilon).
		\end{equation*}
		Then $R_\varepsilon(-\infty)=0$,
		$R_\varepsilon(+\infty)=\beta_\varepsilon$ (see Figure 1), and, after taking
		$|\varepsilon|$ smaller if necessary,
		\begin{equation}\label{eq:explicit-monotone-pressure}
			\sgn R_\varepsilon'(x)=\sgn\beta_\varepsilon
			=-\sgn\varepsilon.
		\end{equation}
		{Consequently, the normalized pressure profile defined by}
		\begin{equation*}
			r_\varepsilon(x):=\frac{R_\varepsilon(x)}{\beta_\varepsilon}
		\end{equation*}
		satisfies Assumption~\ref{ass:pressure}.
	\end{example}
\begin{figure}
	\centering
	\includegraphics[width=0.7\linewidth]{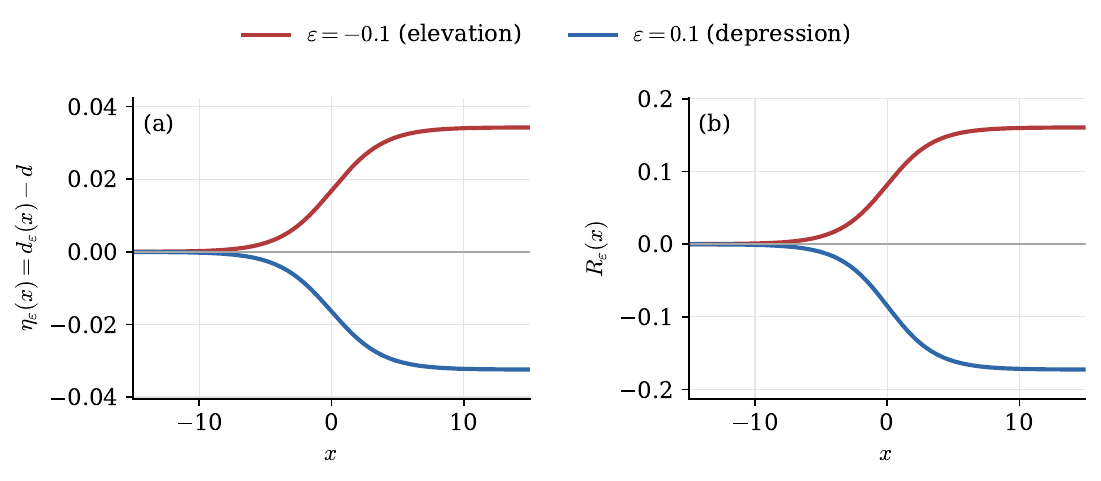}
	\caption{Profiles of the explicit fronts and pressures.}
	\label{fig:constant-vorticity}
\end{figure}

	\begin{proof}[{Verification of the example}]
{The function $\Theta_k$ is harmonic wherever the logarithm is}
		analytic.  We take $k>0$ sufficiently small so that $kd<\pi/2$.
		Since $d_+(\varepsilon)\to d$ as $\varepsilon\to0$, after decreasing
		$|\varepsilon|$ we also have
		\begin{equation}\label{eq:explicit-strip-condition}
			k\max\{d,d_+(\varepsilon)\}<\frac{\pi}{2}.
		\end{equation}
		Thus the expression in \eqref{eq:explicit-theta} is single-valued throughout
		the whole fluid domain.  It follows directly that
		\begin{equation*}
			-\Delta\psi_\varepsilon=\gamma_0,
			\qquad
			\psi_\varepsilon(x,0)=m.
		\end{equation*}
	
		Direct differentiation gives
		\begin{equation}\label{eq:explicit-theta-derivatives}
			(\Theta_k)_x=\frac{ke^{kx}\sin(ky)}{1+2e^{kx}\cos(ky)+e^{2kx}},
			\qquad
			(\Theta_k)_y=\frac{ke^{kx}(e^{kx}+\cos(ky))}{1+2e^{kx}\cos(ky)+e^{2kx}}.
		\end{equation}
		Hence
		\begin{equation}\label{eq:explicit-psi-derivatives}
			(\psi_\varepsilon)_x
			=-\varepsilon\frac{e^{kx}\sin(ky)}{1+2e^{kx}\cos(ky)+e^{2kx}},
			\qquad
			-(\psi_\varepsilon)_y
			=\mu+\gamma_0y
			+\varepsilon\frac{e^{kx}(e^{kx}+\cos(ky))}{1+2e^{kx}\cos(ky)+e^{2kx}}.
		\end{equation}
		Under \eqref{eq:explicit-strip-condition},
		\begin{equation*}
			0<\frac{e^{kx}(e^{kx}+\cos(ky))}{1+2e^{kx}\cos(ky)+e^{2kx}}<1.
		\end{equation*}
		Together with \eqref{eq:explicit-supercritical}, this shows that
		$-(\psi_\varepsilon)_y>0$ for sufficiently small
		$|\varepsilon|$.  The implicit function theorem therefore gives the
		unique surface $y=\eta_\varepsilon(x)$.  Letting $x\to\pm\infty$ in
		\eqref{eq:explicit-stream} gives
		\begin{equation*}
			\Theta_k(x,y)\to0\quad\text{as }x\to-\infty,
			\qquad
			\Theta_k(x,y)\to ky\quad\text{as }x\to+\infty,
		\end{equation*}
		which proves \eqref{eq:explicit-depth-limits} and
		\eqref{eq:explicit-downstream-depth}.  Differentiating
		$\psi_\varepsilon(x,\eta_\varepsilon(x))=0$ and using
		\eqref{eq:explicit-psi-derivatives}, we obtain
		\begin{equation*}
			\eta_\varepsilon'(x)
			=-\frac{(\psi_\varepsilon)_x}
			{(\psi_\varepsilon)_y}
			{\bigg|_{y=\eta_\varepsilon(x)}}.
		\end{equation*}
		Since $0<k\eta_\varepsilon(x)<\pi/2$, this proves
		\eqref{eq:explicit-monotone-surface}.
		
		We now verify the monotonicity of the pressure.  By
		\eqref{eq:explicit-supercritical}, one can choose $k>0$ sufficiently
		small so that, in addition to $kd<\pi/2$,
		\begin{equation}\label{eq:explicit-k-condition}
			ak>\left(\gamma_0+\frac{g}{a}\right)\tan(kd).
		\end{equation}
		Indeed, after dividing by $k$ and letting $k\to0$, this condition
		reduces to $\mu a>gd$.  From the surface equation, uniformly for
		$x\in\R$,
		\begin{equation*}
			\eta_\varepsilon(x)
			=d-\frac{\varepsilon}{ka}\Theta_k(x,d)
			+O(\varepsilon^2).
		\end{equation*}
		{For brevity, set $C:=\gamma_0+g/a$ in the following calculation. Using \eqref{eq:explicit-theta-derivatives} in \eqref{eq:explicit-pressure}, we obtain}
		\begin{equation}\label{eq:explicit-pressure-derivative}
			\begin{aligned}
				R_\varepsilon'(x)
				={}&\varepsilon\frac{e^{kx}}{{(1+2e^{kx}\cos(kd)+e^{2kx})^2}}
				\bigg[\left\{C\sin(kd)-ak\cos(kd)\right\}(1+e^{2kx})\\
				&\hspace{35mm}
				+2e^{kx}\left\{C\sin(kd)\cos(kd)-ak\right\}\bigg]
				+O\left(\varepsilon^2\frac{e^{kx}}{(1+e^{kx})^2}\right),
			\end{aligned}
		\end{equation}
		{The remainder} in \eqref{eq:explicit-pressure-derivative} is uniform
		in $x$.  Condition \eqref{eq:explicit-k-condition} shows that both
		coefficients in the square bracket are strictly negative.  Since
		$1+2e^{kx}\cos(kd)+e^{2kx}$ is comparable with $(1+e^{kx})^2$, we conclude that
		\begin{equation}\label{eq:explicit-R-sign-epsilon}
			\sgn R_\varepsilon'(x)=-\sgn\varepsilon
		\end{equation}
		for every sufficiently small $\varepsilon\ne0$.
		
		Finally, differentiating \eqref{eq:explicit-downstream-depth} at
		$\varepsilon=0$ gives
		\begin{equation*}
			d_+'(0)=-\frac{d}{a}.
		\end{equation*}
		Consequently,
		\begin{equation*}
			\beta_\varepsilon
			=-\varepsilon\frac{\mu a-gd}{a}+O(\varepsilon^2).
		\end{equation*}
		Thus \eqref{eq:explicit-monotone-pressure} follows from
		\eqref{eq:explicit-supercritical} and
		\eqref{eq:explicit-R-sign-epsilon}.  The exponential convergence of
		$r_\varepsilon$ and its derivatives follows directly from
	 \eqref{eq:explicit-stream} and the implicit function
		theorem.  This proves that $r_\varepsilon$ satisfies
		Assumption~\ref{ass:pressure}.
\end{proof}

\end{document}